\documentclass[a4paper,12pt]{article}
\usepackage{amsmath,amsthm,amssymb}
\usepackage[dvipsnames,svgnames,table]{xcolor}
\usepackage[plain]{fullpage}
\usepackage{enumerate} 
\usepackage{subcaption}
\usepackage{comment}

\usepackage{hyperref}

\usepackage{tikz}
\usepackage{authblk,enumitem}
\usepackage{thm-restate}

\newtheorem{theorem}{Theorem}[section]
\newtheorem{corollary}[theorem]{Corollary}
\newtheorem{proposition}[theorem]{Proposition}
\newtheorem{lemma}[theorem]{Lemma}
\newtheorem{claim}{Claim}[theorem]

\theoremstyle{definition}

\newtheorem{problem}[theorem]{Problem}

\definecolor{dark-green}{rgb}{0.2, 0.5, 0.2}

\newcommand{\ora}[1]{\overrightarrow{#1}}
\newcommand{\ola}[1]{\overleftarrow{#1}}

\renewcommand{\vec}[1]{\ora{#1}} %I put larger arrows
\newcommand{\cev}[1]{\ola{#1}} %I put larger arrows

\renewcommand{\leq}{\leqslant}

\renewcommand{\geq}{\geqslant}

\newenvironment{subproof}[1][]
	{\begin{proof}[Proof of the claim] }{\end{proof}}

\DeclareMathOperator{\Inv}{Inv}

\DeclareMathOperator{\inv}{inv}

\DeclareMathOperator{\fas}{fas}

\DeclareMathOperator{\conv}{conv}
\DeclareMathOperator{\id}{id}

\DeclareMathOperator{\ct}{\ct}

\title{On the \texorpdfstring{$(\leq p)$}{(<=p)}-inversion diameter of oriented graphs \thanks{This work was partially supported by the french Agence Nationale de la Recherche under contract Digraphs ANR-19-CE48-0013-01. Caroline Silva was supported by FAPESP Proc. 2023/16755-2.}}

\author[1]{Fr\'ed\'eric Havet}
\author[1]{Cl\'ement Rambaud}
\author[1,2]{Caroline Silva}

\affil[1]{Universit\'e C\^ote d'Azur, CNRS, Inria, I3S, Sophia Antipolis, France}
\affil[2]{Institute of Computing, UNICAMP, Campinas, Brazil}

\date{}

\begin{document}
\maketitle

\begin{abstract}
In an oriented graph $\vec{G}$, the {\it inversion} of a subset $X$ of vertices consists in reversing the orientation of all arcs with both endvertices in $X$.
The {\it $(\leq p)$-inversion graph} of a labelled graph $G$, denoted by ${\mathcal{I}}^{\leq p}(G)$, is the graph whose vertices are the labelled orientations of $G$ in which two labelled orientations $\vec{G}_1$ and $\vec{G}_2$ of $G$ are adjacent if and only if
there is a set $X$ with $|X|\leq p$ whose inversion transforms $\vec{G}_1$ into $\vec{G}_2$.
In this paper, we study the {\it $(\leq p)$-inversion diameter} of a graph, denoted by $\id^{\leq p}(G)$, which is the diameter of its $(\leq p)$-inversion graph.
We show that there exists a smallest number $\Psi_p$ with $\frac{1}{4}p - \frac{3}{2} \leq \Psi_p \leq \frac{1}{2}p^2$ such that $\id^{\leq p}(G) \leq \left\lceil\frac{|E(G)|}{\lfloor p/2\rfloor}\right \rceil + \Psi_p$ for all graph $G$.
We then establish better upper bounds for several families of graphs and in particular trees and planar graphs. Let us denote by $\id^{\leq p}_{\cal F}(n)$ (resp. $\id^{\leq p}_{\cal P}(n)$) the maximum $(\leq p)$-inversion diameter of a tree (resp. planar graph) of order $n$. For trees, we show 
$\id^{\leq 3}_{\cal F}(n) = \left\lceil \frac{n-1}{2}\right\rceil$, $\id^{\leq 4}_{\cal F}(n)=\frac{3}{8}n + \Theta(1)$, $\id^{\leq 5}_{\cal F}(n)= \frac{2}{7}n + \Theta(1)$, and $\id^{\leq p}_{\cal F}(n) \leq \frac{n-1}{p- c\sqrt{p}} + 2$ with $c = \sqrt{2 + \sqrt{2}}$ for all $p\geq 6$.
For planar graphs, we prove  $\id^{\leq 3}_{\cal P}(n) \leq \frac{11n}{6} - \frac{8}{3}$, $\id^{\leq 4}_{\cal P}(n) \leq \frac{4n}{3} + \frac{10}{3}$,
and $\id^{\leq p}_{\cal P}(n) \leq \left\lceil\frac{3n-6}{\lfloor p/2\rfloor}\right \rceil + 8\lfloor p/2\rfloor - 8$ for all $p\geq 6$.

\medskip
    \noindent{}{\bf Keywords:}  inversion graph; diameter; orientation; reconfiguration.
\end{abstract}

\section{Introduction}
%%%%%%%%%%%%%%%%%%%%%%

Making a digraph acyclic by either removing a minimum cardinality set of arcs is an important and heavily studied problem, known as the  {\sc Feedback Arc Set} problem.   
A {\bf feedback arc set}  in a digraph is a set of arcs whose deletion results in an acyclic digraph. The {\bf feedback-arc-set number} of a digraph $D$, denoted by $\fas(D)$, is the minimum size of a feedback arc set. 
Note that if $F$ is a minimum feedback arc-set in a digraph $D=(V,A)$, then we will obtain an acyclic digraph from $D$ by either removing the arcs of $F$ or reversing each of these, that is replacing each arc $uv\in F$ by the arc $vu$.
Computing $\fas(D)$ is one of the first problems shown to be NP-hard listed by Karp in~\cite{karp1972}. It also remains NP-complete in tournaments as shown by Alon~\cite{alonSJDM20} and Charbit, Thomass\'e, and Yeo~\cite{charbitCPC16}.

Belkhechine et al. in~\cite{BBBP10} introduced the more general concept of inversion.
In an oriented graph $\vec{G}$, if $X$ is a set of vertices of $\vec{G}$, the {\bf inversion} of $X$ consists in reversing the orientation of all arcs with both endvertices in $X$.
In particular, the reversal of a single arc corresponds to the inversion of the set of its endvertices.
Belkhechine et al. in~\cite{BBBP10} studied the {\bf inversion number}, denoted by $\inv(D)$, which is the minimum number of inversions that transform $\vec{G}$ into a directed acyclic graph.
In particular, they proved that for every fixed $k$, determining whether a given tournament $T$ has inversion number at most $k$ is polynomial-time solvable.
In contrast, Bang-Jensen et al.~\cite{BCH} proved that deciding whether a given oriented graph has inversion number~$1$ is NP-complete.
The maximum $\inv(n)$ of the inversion numbers of all oriented graphs of order $n$ has also been investigated. Independently, Aubian et al.~\cite{inversion} and Alon et al.~\cite{APSSW} proved
$n - 2\sqrt{n\log n} \leq \inv(n) \leq n - \lceil \log (n+1) \rceil$.

Making the connection between feedback arc set (inversion over sets of size 2) and inversion over sets of unbounded size, Yuster~\cite{yuster2023tournamentinversion} studied {\bf $(\leq p)$-inversions}, which are inversions over sets of size at most $p$. (Bang-Jensen et al.~\cite{bang2025making} also studied $(= p)$-inversions, which are inversions over sets of size exactly $p$.)
Given an oriented graph $\vec{G}$, its {\bf  $(\leq p)$-inversion number}, denoted by $\inv^{\leq p}(\vec{G})$, is the minimum number of $(\leq p)$-inversions rendering $\vec{G}$ acyclic. Note that $\inv^{\leq 2}(\vec{G}) =\fas(\vec{G})$. 
Yuster studied the maximum $\inv^{\leq p}(n)$ of the $(\leq p)$-inversion numbers of all oriented graphs of order $n$. 
Results of Spencer~\cite{spencer1971,spencer1980} (later improved by de la Vega~\cite{delavega1983}) on feedback arc sets show $\inv^{\leq 2}(n) = \frac{1}{2}\binom{n}{2} +\Theta(n^{3/2})$.
Yuster~\cite{yuster2023tournamentinversion} proved $
\frac{1}{12}n^2 + o(n^2) \leq \inv^{\leq 3}(n) \leq \frac{257}{2592}n^2 + o(n^2)$ and conjectured 
$\inv^{\leq 3}(n) = \frac{1}{12}n^2 + o(n^2)$.
He also proved that, for every fixed $p$,
$$
\left(\frac{1}{2p(p-1)}+\delta_p\right)n^2 + o(n^2) \leq \inv^{\leq p}(n)\leq \left( \frac{1}{2\lfloor p^2/2\rfloor} -\epsilon_p\right)n^2 + o(n^2) $$
where $\epsilon_p >0$ for all $p\geq 3$ and $\delta_p >0$ for all $p\geq p_0$ for some $p_0$.

\medskip

Rather than reducing to an acyclic digraph, we can use inversions to reduce to other types of digraphs.
Duron et al.~\cite{duron2023minimum} studied the minimum number of inversions to make a digraph
$k$-arc-strong or $k$-strong.
More generally, Havet et al. \cite{HHR24} studied the maximum over all pairs of orientations of a graph of the minimum number of inversions required to transform one of the orientations into the other.
Formally, let $G$ be a graph with vertices labelled $v_1, \dots , v_n$.
The {\bf (labelled) inversion graph} of $G$, denoted by ${\mathcal{I}}(G)$, is a graph with vertex set the labelled orientations of $G$. Then, two labelled orientations $\vec{G}_1$ and $\vec{G}_2$ of $G$ are adjacent if and only if there is an inversion $X$ transforming $\vec{G}_1$ into $\vec{G}_2$.
The {\bf (labelled) inversion diameter} of $G$ is the diameter of ${\mathcal{I}}(G)$, denoted by 
$\id(G)$. 
 Havet et al.~\cite{HHR24} determined the maximum inversion diameter over all graphs on $n$ vertices. 
 
\begin{theorem}\label{thm:diam-upper}
    Let $G$ be a graph on $n$ vertices.
    Then $\id(G) \leq n-1$.
\end{theorem}
This bound is tight for complete graphs. 
For sparser graphs, Havet et al.~\cite{HHR24} showed that much better bounds can be obtained. Some of their bounds were subsequently improved by Arana et al.~\cite{arana2026}.
\begin{enumerate}[label=\alph*)]
    \item $\id(G) \leq 12$ for every planar graph $G$ \cite{HHR24},
        and there are planar graphs with  inversion diameter at least $6$ \cite{arana2026};
    \item $\id(G) \leq 3$ for every planar graph $G$ of girth at least $7$ \cite{arana2026},
        and there are planar graphs of arbitrary large girth with inversion diameter $3$  \cite{HHR24};
     \item $\id(G) \leq 2\Delta(G)-1$ for every graph $G$ with at least one edge
        and $\id(G)=\Delta(G)$ for every complete graph $G$  \cite{HHR24};
    \item $\id(G) \leq 2t$ for every graph $G$ of treewidth at most $t$~\cite{HHR24} and for every positive integer $t$, there are graphs of treewidth $t$ with  inversion diameter $2t$ ~\cite{WWYL25}.
\end{enumerate}

The {\bf (labelled) $(\leq p)$-inversion graph} of $G$, denoted by ${\mathcal{I}}^{\leq p}(G)$, is the graph whose vertices are the labelled orientations of $G$ in which two labelled orientations $\vec{G}_1$ and $\vec{G}_2$ of $G$ are adjacent if and only if there is an $(\leq p)$-inversion transforming $\vec{G}_1$ into $\vec{G}_2$.
The {\bf (labelled) $(\leq p)$-inversion diameter} of $G$ is the diameter of ${\mathcal{I}}^{\leq p}(G)$, denoted by $\id^{\leq p}(G)$.
The {\bf $(\leq p)$-converse number} of a graph $G$, denoted by $\conv^{\leq p}(G)$, which is the minimum number of $(\leq p)$-inversions to transform an orientation of $G$ into its converse. Note that this number does not depend on the orientation. 
An $(\leq p)$-inversion reverses at most $\binom{p}{2}$ arcs, and one can reverse precisely one arc by inverting the set of its endvertices, thus 
\begin{equation}\label{eq:easy}
    \frac{|E(G)|}{\binom{p}{2}}\leq \conv^{\leq p}(G) \leq \id^{\leq p}(G)\leq |E(G)|.
\end{equation}

For $p=2$, the above equation yields the equality $\conv^{\leq p}(G) =  \id^{\leq p}(G) = |E(G)|$.
The left-hand side inequality of Equation~\eqref{eq:easy} is attained for very sparse graphs, for example the disjoint union of complete graphs of order $p$.
In contrast, the right-hand side inequality  is not tight when $p > 2$. 
In Section~\ref{sec:upper-gen}, using strong edge colourings, we show the following upper bound.  

\begin{restatable}{theorem}{uppergen}
\label{thm:uppergen}
    Let $G$ be a graph and $p\geq2$ be an integer. Then $\id^{\leq p}(G) \leq \left\lceil\frac{|E(G)|}{\lfloor p/2\rfloor}\right \rceil + \frac{1}{2} p^2$.
\end{restatable}

This upper bound is tight up to the additive term $\frac{1}{2} p^2$.
Indeed, consider a matching graph $M$, that is, the disjoint union of $K_2$ and $K_1$. Then, an $(\leq p)$-inversion reverses at most $\lfloor p/2\rfloor$ edges, so $\conv^{\leq p}(M)=  \id^{\leq p}(M) = \left\lceil \frac{|E(M)|}{\lfloor p/2\rfloor}\right \rceil$.
In particular, for $p\in \{2,3\}$, we get $\conv^{\leq p}(M) =  \id^{\leq p}(M) = |E(M)|$, so the right-hand side
inequalities of Equation~\ref{eq:easy} are tight.

% Hence, a natural question 
% is to determine the smallest number $\Psi_p$, (resp. $\Psi'_p$) such that 
% $\id^{\leq p}(G)\leq \left\lceil \frac{|E(G)|}{\lfloor p/2\rfloor}\right \rceil + \Psi_p$
% (resp. $\conv^{\leq p}(G) \leq  \left\lceil \frac{|E(G)|}{\lfloor p/2\rfloor}\right \rceil + \Psi'_p$) 
% for every graph $G$.

% Observe that if $n\leq p$, then $\id^{\leq p}(K_n) = \id(K_n) = n-1$ as proved in \cite{HHR24},
% so $\Psi_p \geq n-1 - \left\lceil \frac{\binom{n}{2}}{\lfloor p/2\rfloor}\right \rceil \geq n-1 - \left\lceil \frac{n(n-1)}{p-1} \right\rceil$.
% For $n=\lceil p/2 \rceil$, we obtain $\Psi_p \geq \frac{1}{4}p -\frac{3}{2}$.

% We believe that this lower bound is tighter than the upper bound.

% \begin{problem}\label{prob:m/2}
%  Does there exist a constant  $C$ such that $\Psi_p \leq C\cdot p$ for every integer $p$ greater than $1$?
% \end{problem}

The upper bound of Theorem~\ref{thm:uppergen} is tight for matching graphs, which are very sparse. 
It is however far from being tight for not too sparse graphs. 
In Theorem~\ref{thm:procedure1}, we prove the following upper bound.

\begin{equation}\label{eq:m/p-1}
\id^{\leq p}(G) \leq \frac{1}{p-1}|E(G)| + \frac{p-2}{p-1}(|V(G)| -1).
\end{equation}

Note that this bound is better than the one of Equation~\eqref{eq:easy} if $|E(G)|\geq |V(G)| -1$ and better than the one of Theorem~\ref{thm:uppergen}
if $|E(G)|\geq (p-2)(|V(G)| -1)$ when $p$ is odd and if 
 $|E(G)|\geq p(|V(G)| -1)$ when $p$ is even.

We can get an even better upper bound for very dense graphs (graphs of order $n$ with $\Omega (n^2)$ edges). Using a classical theorem by K{\H{o}}v{\'a}ri, S{\'o}s, and Tur{\'a}n~\cite{kovariCM3}, Yuster~\cite{yuster2023tournamentinversion} implicitely used the following theorem, which intuitively, means that after $\frac{|E|}{\lceil p/2\rceil \cdot \lfloor p/2\rfloor}$  inversions of well-chosen sets of size $p$,
we have reversed exactly the edges in $E$, up to $o(n^2)$ errors
(See Bang-Jensen et al.~\cite{bang2025making} for an explicit proof).

\begin{theorem}\label{thm:reversing_a_dense_set_of_edges}
    Let $p$ be an integer with $p \geq 2$.
    There exist constants $\alpha_p$ and $\epsilon_p$ with $\epsilon_p>0$ such that the following holds.
    Let $n$ be a positive integer,  and let $F \subseteq \binom{[n]}{2}$.
    There exists an integer $\ell$ with $\ell \leq \frac{|F|}{\lceil p/2\rceil \cdot \lfloor p/2\rfloor}$,
    and a family $X_1, \dots, X_\ell$ of $(=p)$-subsets of $[n]$ such that
    \[
        \left|F \Delta \binom{X_1}{2} \Delta \dots \Delta \binom{X_\ell}{2}\right| \leq \alpha_p \cdot n^{2-\epsilon_p}.
    \]
\end{theorem}

This theorem directly implies

\begin{equation}\label{eq:m/p2}
\id^{\leq p}(G) \leq \frac{|E(G)|}{\lceil p/2\rceil \cdot \lfloor p/2\rfloor} + o(|V(G)|^2).
\end{equation}

This bound is asymptotically tight. Indeed, consider two orientations of a complete bipartite graph $B$ which are converse to each other. Every inversion reverses at most $\lceil p/2\rceil \cdot \lfloor p/2\rfloor$ edges. Since all edges of $B$ need to be reversed, we get
$\id^{\leq p}(B) \geq \frac{|E(B)|}{\lceil p/2\rceil \cdot \lfloor p/2\rfloor}$.

\medskip

In Section~\ref{sec:upper-connected}, we give upper bounds of the $(\leq 3)$-inversion diameter of a connected graph in terms of its triangle-transversal numbers and its triangle-edge-packing number and derive the following corollary.

\begin{theorem}\label{thm:connected}
  If $G$ is a connected graph, then  $\id^{\leq 3}(G) 
  \leq \left\lceil \frac{3|E(G)|}{4} \right\rceil $.
\end{theorem}

Let $\mathcal{C}$ be a graph class. It is {\bf $(a,b)$-sparse} if
$|E(G)| \leq a|V(G)| + b$ for every $G\in \mathcal{C}$, and {\bf sparse} if it is $(a,b)$-sparse for some pair $(a,b)$.
 The {\bf $(\leq p)$-inversion diameter function}, of  $\mathcal{C}$, denoted by $\id^{\leq p}_{\mathcal{C}}$, is the smallest function $f$ such that $\id^{\leq p}(G) \leq f(|V(G)|)$ for every $G\in \mathcal{C}$.
 Similarly,  the {\bf $(\leq p)$-converse function}, of  $\mathcal{C}$, denoted by $\conv^{\leq p}_{\mathcal{C}}$, is the smallest function $g$ such that $\conv^{\leq p}(G) \leq g(|V(G)|)$ for every $G\in \mathcal{C}$.
Clearly, $\conv^{\leq p}_{\cal C}(n) \leq \id^{\leq p}_{\cal C}(n)$ for every $n$.
If $\cal C$ is $(a,b)$-sparse, then Theorem~\ref{thm:uppergen} and  Equation~\eqref{eq:m/p-1} yield the upper bound 
$$\conv^{\leq p}_{\cal C}(n) \leq \id^{\leq p}_{\cal C}(n) \leq \min \left \{\left\lceil\frac{an+b}{\lfloor p/2\rfloor}\right \rceil +\frac{1}{2}p^2 ,  \frac{a+p+2}{p-1}n + \frac{b-p+2}{p-1}\right\}.$$
In the remaining of the paper, we improve on this bound for several well-known sparse classes of  graphs.

In Section~\ref{sec:trees}, we determine the $(\leq p)$-inversion diameter function and $(\leq p)$-converse function of the class ${\cal F}$ of forests up to an additive constant when $p\leq 5$ and gives an upper bound on those parameters which improves on the above upper bounds for any value of~$p$ greater than $4$.

\begin{theorem}\label{thm:trees}
\begin{itemize}
\item[(i)] $\id^{\leq 3}_{\cal F}(n) = \conv^{\leq 3}_{\cal F}(n)= \left\lceil \frac{n-1}{2}\right\rceil$.   
\item[(ii)]  $\id^{\leq 4}_{\cal F}(n), \conv^{\leq 4}_{\cal F}(n)=\frac{3}{8}n + \Theta(1)$.
\item[(iii)]  $\id^{\leq 5}_{\cal F}(n), \conv^{\leq 5}_{\cal F}(n)= \frac{2}{7}n + \Theta(1)$.
\item[(iv)] $\conv^{\leq p}_{\cal F}(n)  \leq \id^{\leq p}_{\cal F}(n) \leq \frac{n-1}{p- c\sqrt{p}} + 2$ with $c = \sqrt{2 + \sqrt{2}}$.

\end{itemize}
\end{theorem}

In Section~\ref{sec:k-deg}, we show that if 
$G$ is a $k$-degenerate graph, then 
$\id^{\leq p}(G) \leq |V(G)| -1$, for any $p\geq k+1$.
We show that it is tight when $k=2$ and $p=3$.

\medskip

In Section~\ref{sec:planar}, we consider the class ${\cal P}$ of planar graphs and show $\frac{7n}{6}-2 \leq \conv^{\leq 3}_{\cal P}(n) \leq \id^{\leq 3}_{\cal P}(n) \leq \frac{11n}{6}-\frac{8}{3}$, $\id^{\leq 4}_{\cal P}(n) \leq \frac{4n}{3}+\frac{10}{3}$, and $\frac{p-1}{(p-2)^2+1} \cdot (n-2) \leq  \conv^{\leq p}_{\cal P}(n) \leq \id^{\leq p}_{\cal P}(n) \leq \left\lceil\frac{3n-6}{\lfloor p/2\rfloor}\right \rceil + 8\lfloor p/2\rfloor - 8$ for every $p\geq 4$.
%It would be interesting to close the gaps between those upper and lower bounds.

\section{Notations, definitions, and preliminaires}
%%%%%%%%%%%%%%%%%%%%%%%%%%%%%%%%%%%%%%%%%%%%%%%%%%
Notation not given below is consistent with \cite{bang2009}. A {\bf $(\leq p)$-set} is a set of cardinality at most $p$. 
Let $D$ be an oriented graph and $\cal X$ be a family of subsets of $V(D)$.
We say that $\cal X$ is an {\bf $(\leq p)$-family} if all members of $\cal X$ are $(\leq p)$-sets.
We denote by $\Inv(D; {\cal X})$ the oriented graph obtained after inverting all sets of $\cal X$ one after another. Observe that this is independent of the order in which we invert those sets: $\Inv(D; {\cal X})$ is obtained from $D$ by reversing exactly those arcs for which an odd number of members of ${\cal X}$ contain both endvertices.
If ${\cal X} = \{X\}$ for a set $X\subseteq V(D)$, then we write $\Inv(D; X)$ for $\Inv(D; {\cal X})$.

Let $\vec{G}_1$ and $\vec{G}_2$ be two orientations of a graph $G$.
If an edge $e$ has the same orientation in $\vec{G}_1$ and $\vec{G}_2$, we say that $\vec{G}_1$ and $\vec{G}_2$ {\bf agree} on $e$;
otherwise we say that they {\bf disagree} on $e$. 
We denote by $E_{=}$ the set of edges of $G$ on which $\vec{G}_1$ and $\vec{G}_2$ agree 
and by $E_{\neq}$ the set of edges of $G$ on which $\vec{G}_1$ and $\vec{G}_2$ disagree.
We set $G_{\neq} = (V(G), E_{\neq})$.
A {\bf $k$-vertex} in a graph $G$ is a vertex of degree $k$.

\begin{proposition}\label{prop:H+I}
    Let $G$ be a graph.
If $G$ is the union of a subgraph $H$ and an induced subgraph $I$, then
$\id^{\leq p}(G) \leq  \id^{\leq p}(H) + \id^{\leq p}(I)$.
\end{proposition}
\begin{proof}
Since $\id^{\leq p}$ is monotone, free to remove some edges of $H$, we may assume that
$E(H)\cap E(I) =\emptyset$.
%Set $t_1= \id^{\leq p}(H)$ and $t_2=\id^{\leq p}(I)$.

Let $\vec{G}_1$ and $\vec{G}_2$ be two orientations of $G$.
For $i \in \{1,2\}$, let $\vec{H}_i$ be the orientation of $H$ which agrees with  $\vec{G}_i$ on $V(H)$.
There exists a $(\leq p)$-family ${\cal X}_H$ of size at most $\id^{\leq p}(H)$ such that $\Inv(\vec{H}_1; {\cal X}_H) = \vec{H}_2$.
Set $\vec{G}_3 = \Inv(\vec{G}_1; {\cal X}_H)$. 
Clearly, $\vec{G}_3$ agrees with $\vec{G}_2$ on $E(H)$.
For $i\in\{2,3\}$, let $\vec{I}_i$ be the orientation of $I$ induced by  $\vec{G}_i$.
There exists a $(\leq p)$-family ${\cal X}_I$ of size at most $\id^{\leq p}(I)$ such that $\Inv(\vec{I}_3; {\cal X}_I) = \vec{I}_2$.
Each element of ${\cal X}_I$ is a subset of $V(I)$ and, thus, does not contain any pair of endvertives of edges of $H$, since $I$ is an induced subgraph and $E(H)\cap E(I) =\emptyset$.
Therefore, no edges of $H$ is reversed by the inversion of ${\cal X}_I$.
Hence, $\Inv(\vec{G}_1; {\cal X}_H\cup {\cal X}_I) = \Inv(\vec{G}_3; {\cal X}_I) = \vec{G}_2$.
Thus, $\vec{G}_1$ and $\vec{G}_2$ are at distance at most
$|{\cal X}_H| + |{\cal X}_I| \leq \id^{\leq p}(H) + \id^{\leq p}(I)$ in ${\mathcal{I}}^{\leq p}(G)$.
\end{proof}

A matching $M$ of a graph $G$ is an \textbf{induced matching} of $G$ if every edge connecting any two endvertices of edges of $M$ is in $M$.
Since a graph whose edge set is a matching of size at most $\lfloor p/2 \rfloor$ has $(\leq p)$-inversion diameter $1$, Proposition~\ref{prop:H+I} immediately yields the following.

\begin{corollary}\label{cor:induced-matching}
        Let $p \geq 2$ be an integer. 
        Let $G$ be a graph and let $M$ be an induced matching of $G$ of size at most $\lfloor p/2 \rfloor$. Then $\id^{\leq p}(G) \leq \id^{\leq p}(G \setminus M) + 1$.
    \end{corollary}

\begin{corollary}\label{cor:G-v}
        Let $p \geq 2$ be an integer and set $q = \lfloor p/2\rfloor$.  
        Let $G$ be a graph and $v$ a vertex of $G$.
\begin{itemize}        
     \item [(i)] $\id^{\leq p}(G) \leq \id^{\leq p}(G-v) + \left\lceil \frac{d(v)}{p-1} \right\rceil$.

     \item [(ii)] If $d(v) \geq q$, 
        then $\id^{\leq p}(G) \leq \id^{\leq p}(G-v) + \left\lfloor d(v) / q\right\rfloor$.
\end{itemize}        
    \end{corollary}
\begin{proof}
Let $H$ be the subgraph of $G$ induced by the edges incident to $v$. 

\medskip

(i) $H$ is a star with $d(v)$ leaves, which is the union of $\lceil \frac{d(v)}{p-1}\rceil$ edge-disjoint substars of order at most $p$. Each of these substars has $(\leq p)$-inversion diameter $1$, so $\id^{\leq p}(H) \leq \lceil \frac{d(v)}{p-1}\rceil$.
Now $G$ is the union of $H$ and $G-v$, so by Proposition~\ref{prop:H+I}, $\id^{\leq p}(G) \leq \id^{\leq p}(H) + \id^{\leq p}(G-v) \leq \id^{\leq p}(G-v) + \lceil \frac{d(v)}{p-1}\rceil$.
  
  \medskip
  
  (ii) 
    Let $\vec{H}_1$ and $\vec{H}_2$ be two orientations of $H$. 
    Then $N(v)$ can be partitioned into $t=\lfloor d(v)/q \rfloor$ sets $Y_1, \dots, Y_t$ of size at least $q$ and at most $2q-1$.
    For every $i\in [t]$, let $X_i = \{v\} \cup \{w\in Y_i \mid vw \in E_{\neq}\}$. Clearly, $|X_i| \leq |Y_i| +1 \leq 2q \leq p$. Thus $(X_i)_{i\in [t]}$ is an $(\leq p)$-family and $\Inv(\vec{H}_1 ; (X_i)_{i\in [t]}) = \vec{H}_2$. Thus $\id^{\leq p}(H) \leq \left\lfloor d(v) / q\right\rfloor$.

    Now $G$ is the union of $H$ and $G-v$, so by Proposition~\ref{prop:H+I}, $\id^{\leq p}(G) \leq \id^{\leq p}(H) + \id^{\leq p}(G-v) \leq \id^{\leq p}(G-v) + \left\lfloor d(v) / q\right\rfloor$.
\end{proof}

\section{General upper bounds}\label{sec:upper-gen}
%%%%%%%%%%%%%%%%%%%%%%%%

\begin{comment}

\begin{lemma}\label{lem:induced-matching}
        Let $p \geq 2$ be an integer.
        Let $G$ be a graph and let $M$ be an induced matching of $G$ of size at most $\lfloor p/2 \rfloor$. Then $\id^{\leq p}(G) \leq \id^{\leq p}(G \setminus M) + 1$.
    \end{lemma}
\begin{proof}
    Let $\vec{G}_1$ and $\vec{G}_2$ be two orientations of $G$.
    Let $M$ be an induced matching of $G$ of size at most $\lfloor p/2 \rfloor$. Let $G' = G - M$. Let $\mathcal{X}$ be a family of at most $\id^{\leq p}(G \setminus M)$ inversions such that $\Inv(\vec{G}_1 \setminus M; \mathcal{X}) = \vec{G}_2 \setminus M$. Let $\vec{H} = \Inv(\vec{G}_1, X)$ and let $M'$ be the set of arcs that disagree in $\vec{H}$ and $\vec{G}_2$. Note that $M' \subseteq M$. Let $Z$ be the set of vertices that are covered by $M'$. Note that, since $|M| = \lfloor p/2\rfloor$, it follows that $|Z| \leq p$. Also, $E(G[Z]) = M'$. Now $\mathcal{X'} = \mathcal{X} \cup \{Z\}$ is a family of $(\leq p)$-inversions such that $\Inv(\vec{G}_1; \mathcal{X'}) = \vec{G}_2$ and $|\mathcal{X'}| = |\mathcal{X}| + 1$. 
\end{proof}
\end{comment}

\begin{lemma}\label{lem:min-ce}
    Let $p \geq 2$ be an integer. Let $G$ be a graph such that every proper subgraph $G'$ of $G$ satisfies $\id^{\leq p}(G') \leq \left\lceil\frac{|E(G')|}{\lfloor p/2\rfloor}\right \rceil + \theta_p$ for some integer $\theta_p$. Suppose that $G$ satisfies one of the following properties:
    \begin{itemize}
        \item $\Delta(G) \geq \lfloor p/2\rfloor$, or 
        \item $G$ contains an induced matching of size $\lfloor p/2\rfloor$.
    \end{itemize}
    Then, $\id^{\leq p}(G) \leq \left\lceil\frac{|E(G)|}{\lfloor p/2\rfloor}\right \rceil + \theta_p$.
\end{lemma}
\begin{proof}
Set $q = \lfloor p/2\rfloor$.  
    Suppose first that $G$ contains a vertex $v$ of degree at least $q$.     
    By Corollary~\ref{cor:G-v}~(ii, $\id^{\leq p}(G) \leq \id^{\leq p}(G-v) + \left\lfloor d(v) / q\right\rfloor$.

    Moreover, by hypothesis, $\id^{\leq p}(G-v) \leq \left\lceil\frac{|E(I)|}{q}\right \rceil + \theta_p = \left\lceil\frac{|E(G) - d(v)|}{q}\right \rceil + \theta_p$. Hence, 

    \[
        \id^{\leq p}(G) \leq \left\lfloor\frac{d(v)}{q}\right\rfloor + \left\lceil\frac{|E(G) - d(v)|}{q}\right \rceil + \theta_p \leq \left\lceil\frac{|E(G)|}{q}\right \rceil + \theta_p.
    \]
    
    Suppose now that $G$ contains an induced matching $M$ of size $q$. Let $G' = G \setminus M$. By hypothesis, $\id^{\leq p}(G') \leq \left\lceil\frac{|E(G)| - q}{q}\right \rceil + \theta_p$. Thus, by Corollary~\ref{cor:induced-matching},
    \[
    \id^{\leq p}(G) \leq \left\lceil\frac{|E(G)| - q}{q}\right \rceil + \theta_p + 1 \leq \left\lceil\frac{|E(G)|}{\lfloor p/2\rfloor}\right \rceil + \theta_p.\qedhere
    \]
\end{proof}

A \textbf{strong edge-colouring} of $G$ is an edge-colouring of $G$ such that each colour class is an induced matching of $G$.
The \textbf{strong chromatic index} of a graph $G$, denoted by $\chi_s(G)$, is the minimum integer $k$ such that $G$ admits a strong edge-colouring with $k$ colours. One can easily show that $\chi_s(G) \leq 2\Delta(G)^2$. However, Erdős and Nešetřil~\cite{faudree1989induced} conjectured that $\chi_s(G) \leq 1.25\Delta(G)^2$. In 1997, Molloy and Reed~\cite{molloy1997bound} proved that there exists $\epsilon > 0$ such that $\chi_s(G) \leq (2-\epsilon)\Delta(G)^2$ for sufficiently large $\Delta(G)$. 
This result was improved by Bonamy et al.~\cite{bonamy2022colouring} and subsequently by Hurley, de Joannis de Verclos and Kang~\cite{Hurley2022} who showed that $\chi_s(G) \leq 1.772\Delta(G)^2$ for sufficiently large $\Delta(G)$.

\begin{lemma}
Let $G$ be a graph and let $p\geq2$ be an integer. Then, $\id^{\leq p}(G) \leq \frac{|E(G)|}{\lfloor p/2 \rfloor} + \chi_s'(G)$.    
\end{lemma}
\begin{proof}
    Let $G$ be a minimum counterexample for the statement. Let $\mathcal{S}$ be a minimum strong edge-colouring of $G$.
    By Lemma~\ref{lem:min-ce}, $G$ has no induced matching of size $\lfloor p/2 \rfloor$. This implies that for every $S \in \mathcal{S}$, $|S| < \lfloor p/2 \rfloor$. An easy induction using Corollary~\ref{cor:induced-matching} shows that $\id^{\leq p}(G) \leq \chi_s'(G)$, a contradiction. 
\end{proof}

% \begin{lemma}
%     Let $G$ be a graph and let $p\geq2$ be an integer. Then, $\id^{\leq p}(G) \leq \frac{|E(G)|}{\lfloor p/2 \rfloor} + \chi_s'(G)$.
% \end{lemma}
% \begin{proof}
% Set $q = \lfloor p/2 \rfloor$.
% Let $\vec{G}_1$ and $\vec{G}_2$ be two orientations of $G$ and let $E_{\neq}$ be the set of edges on which those two orientations disagree.
% Let $k = \chi_s'(G)$. Let $\{S_1, \ldots, S_k\}$ be a minimum strong edge-colouring of $G$. 
% Note that we can decompose $E(G)$ into $t = \sum_{i = 1}^k \left\lceil \frac{|S_i|}{q} \right\rceil$ pairwise disjoint induced matchings $M_1, \ldots, M_t$ of size at most $q$. Let $X_i$ be the set of vertices induced by $M_i \cap E_{\neq}$. So $|X_i| \leq p$,
% $\Inv(\vec{G}_1 ; (X_i)_{i\in [t]}) =\vec{G}_2$ and
% $|(X_i)_{i\in [t]}| \leq t = \sum_{i = 1}^k \left\lceil \frac{|S_i|}{q} \right\rceil 
% \leq \frac{|E(G)|}{q} + k$.
% \end{proof}

%\FH{I would not speak about $\epsilon$ in the two following statement and just add a remark that the bound on $b$ can be imprved using the reult of Molly Reed and Bonamy et al.}

\begin{corollary}\label{cor:bound-Delta}
    Let $G$ be a graph and let $p\geq2$ be an integer. Then, $\id^{\leq p}(G) \leq \frac{|E(G)|}{\lfloor p/2 \rfloor} + 2\Delta(G)^2$.
\end{corollary}

\uppergen*

\begin{proof}
Let $G$ be minimal counterexample for the statement. By Lemma~\ref{lem:min-ce}, $\Delta(G) \leq q =\lfloor p/2\rfloor - 1$. Then, the result follows directly by Corollary \ref{cor:bound-Delta}.
\end{proof}

Observe that the upper bound of $\frac{1}{2}p^2$ in the statement of the previous two results can be slightly improved using the results of~Hurley, de Joannis de Verclos and Kang~\cite{Hurley2022}.

\subsection{Exact values of \texorpdfstring{$\Psi_p$}{Psi\_p}, when \texorpdfstring{$p$}{p} is small}
%%%%%%%%%%%%%%%%%%%%%%%%%%%%%%%%%%%%%

\begin{theorem}\label{thm:psi}
    Let $p$ be an integer with $p \geq 2$. 
    \begin{itemize}
        \item[(i)] If $p\leq 5$, then $\Psi_p =0$.
        
        \item[(ii)] If $p\in\{6,7,8,9\}$, then  $\Psi_p =1$.
    \end{itemize}
\end{theorem}

\begin{proof}
We first show that $\id^{\leq p}(G) \geq \left\lceil\frac{|E(G)|}{2} \right\rceil$ for $p \in \{2,3,4,5\}$ and $\id^{\leq p}(G) \geq \frac{|E(G)|}{2} + 1$ for $p \in \{6,7,8,9\}$.

Consider first a matching graph $G$. Note that $\conv^{\leq p}(G) = \left\lceil \frac{E(G)}{2} \right\rceil$ for $p \in \{2, 3\}$. 
Thus, $\id^{\leq p}(G) \geq  \left\lceil\frac{|E(G)|}{\lfloor p/2 \rfloor}\right\rceil$ for $p \in \{2,3\}$.

Consider now a graph $G$ isomorphic to $K_3$. 
Note that $\id^{\leq p}(K_3) = 2$ for every $p \geq 2$, since reversing exactly two edges of $G$ requires two inversions. 
Since $\left\lceil\frac{|E(G)|}{\lfloor p/2 \rfloor} \right\rceil = 2$ for $p \in \{4, 5\}$, 
it follows that $\id^{\leq p}(G) \geq \left\lceil \frac{|E(G)|}{\lfloor p/2 \rfloor} \right\rceil$. 
Similarly, since $\left\lceil \frac{|E(K_3)|}{\lfloor p/2 \rfloor} \right\rceil = 1$ for $p \geq 6$,
it follows that $\id^{\leq p}(G) \geq \left\lceil \frac{|E(G)|}{\lfloor p/2 \rfloor} \right\rceil + 1$ for $p \in \{6,7,8,9\}$

\medskip

Let us now prove the opposite inequalities, which mean that, for any graph $G$,
$\id^{\leq p}(G) \leq \left\lceil\frac{|E(G)|}{\lfloor p/2 \rfloor}\right \rceil + \Psi_p$, 
with $\Psi_p =0$ if $p \in \{2,3,4,5\}$ and $\Psi_p = 1$ if $p\in \{6,7,8,9\}$.
Note that it suffices to show the result for even values of $p$, that is, $p \in \{ 2,4,6,8\}$. We set $q = p/2$.

We do it by considering a minimum counterexample $G$. 
By Lemma~\ref{lem:min-ce}, we may assume that $\Delta(G) < q$.
This gives a contradiction when $p=2$, as an edgeless graph is clearly no counterexample. 
If $p = 4$, then $\Delta(G) \leq 1$, so $G$ is a matching and $\id^{\leq 4}(G) \leq \left\lceil \frac{|E(G)|}{2} \right \rceil$, a contradiction to $G$ counterexample.
So assume $p \in \{6, 8\}$.

\begin{claim}\label{claim:2sommet}
    If two vertices $x$ and $y$ are at distance at least $3$, then $d(x) +d(y) < q$.  
\end{claim}

\begin{subproof}
Assume for a contradiction that $x$ and $y$ are at distance at least 3 and $d(x) +d(y) \geq q$.
Let $X =\{x,y\} \cup \{w\in N(x) \mid xw \in E_{\neq}\} \cup \{z\in N(y) \mid yz \in E_{\neq}\}$. By Lemma~\ref{lem:min-ce}, $|X| \leq d(x) + d(y) + 2 \leq 2(q-1) + 2 \leq p$.

$\Inv(\vec{G}_1 ; X)$ and $\vec{G}_2$ agree on all arcs incident to $x$ and $y$.
Let $H=G-\{x,y\}$, $\vec{H}_1 = \Inv(\vec{G}_1 ; (X)) -\{x,y\}$ and $\vec{H}_2= \vec{G}_2 -\{x,y\}$.
By the minimality of $G$, there is a family $\mathcal{Z}$ of at most 
$\left\lceil\frac{|E(H)|}{q}\right \rceil+ \Psi_p$ sets of size at most $p$ whose inversion transforms $\vec{H}_1$ into $\vec{H}_2$.
Now $\Inv(\vec{G}_1 ; \{X\}\cup \mathcal{Z}) =\vec{G}_2$,
and $|\{X\}\cup \mathcal{Z}| \leq \left\lceil\frac{|E(H)|}{q}\right \rceil +\Psi_p + 1 \leq \left\lceil\frac{|E(G)|}{q}\right \rceil +\Psi_p$, a contradiction.
\end{subproof}

Assume $p =6$. Lemma~\ref{lem:min-ce} yields $\Delta(G) \leq 2$, so $G$ is the disjoint union of paths and cycles.
Since $G$ is not a matching graph, it has a component $C$ of $G$ is a cycle or a path of order at least $3$. This component has a vertex $v$ of $C$ has degree $2$. Thus, by Claim~\ref{claim:2sommet}, all vertices are distance at least $3$ from $v$ have degree $0$, and so are isolated. But $G$ has no isolated vertices, thus
$G$ is either a path or a cycle of order at most $5$. 
Hence $\id^{\leq p}(G) \leq 2 \leq \left\lceil\frac{|E(G)|}{\lfloor p/2\rfloor}\right \rceil +1$, a contradiction.

\medskip

Assume $p=8$, so $q=4$
and $\Delta(G) \leq 3$.
If $\Delta(G) = 1$, then $G$ is a matching graph and
$\id^{\leq p}(G) \leq \left\lceil\frac{|E(G)|}{\lfloor p/2\rfloor}\right \rceil$, a contradiction.

Assume now $\Delta(G) = 2$. 
Every vertex at distance at least $3$ from a $2$-vertex has degree at most $1$.
Thus $G$ is a path or a cycle of order at most $5$.
Hence $\id^{\leq p}(G) \leq 2 \leq \left\lceil\frac{|E(G)|}{\lfloor p/2\rfloor}\right \rceil +1$, a contradiction. 

Assume finally $\Delta(G) = 3$. 
Let $v$ be a $3$-vertex. By Claim~\ref{claim:2sommet}, every vertex at distance at least $3$ from $v$ has degree $0$, which is impossible in a minimum counterexample. So all vertices of $G$ are at distance at most $2$ from $v$.
Moreover, by Claim~\ref{claim:2sommet}, two $2$-vertices are at distance at most $2$.

\begin{itemize}
\item Assume $G$ has at least two $3$-vertices.

If no two $3$-vertices are adjacent, then $G$ it is the complete bipartite graph $K_{2,3}$. As $\id^{\leq 8}(K_{2,3}) =  \id(K_{2,3}) = 2$, we get a contradiction.

Henceforth, $G$ has two adjacent $3$-vertices, $v$ and $w$. Then the vertices of $G$ are  vertices $v$, $w$, the two neighbours $v_1, v_2$ of $v$ distinct from $w$ and the two neighbours $w_1, w_2$ of $w$ distinct from $v$ with possibly
$v_1=w_1$ and $v_2=w_2$.
So $|E(G)|\geq 5$, and $\left\lceil\frac{|E(G)|}{\lfloor p/2\rfloor}\right \rceil +1 \geq 3$.
If $v_1=w_1$ and $v_2=w_2$, then either $G$ is 
isomorphic to $K_4$ or $G$ is a subgraph of $G'$ where $V(G') = \{v, w, v_1, v_2, t\}$ and $E(G')=\{vw, vv_1, vv_2, wv_1, wv_2, v_1t, v_2t\}$.
In the first case, $\id(G) = \id(K_4) = 3$, a contradiction. 
In the second case, observe that $\id^{\leq p}(G) \leq \id^{\leq p}(G')$. Moreover, $G'$ is the union of two stars of order $4$ with centers $v_1$ and $v_2$ and edge $vw$ which forms an induced $K_2$. Those three graphs have $(\leq p)$-inversion number $1$, thus, by Proposition~\ref{prop:H+I} applied twice, we get $\id^{\leq p}(G) \leq \id^{\leq}(G') \leq 3$, a contradiction.

If $v_1=w_1$ and $v_2\neq w_2$, then by the above arguments, the only possible edge incident to neither $v$ nor $w$ is $v_2w_2$.
Thus $G$ is the union of the tree $T=G\setminus \{vv_1, vv_2\}$ and the induced subgraph $I=G\langle \{v,v_1, v_2\}\rangle$.
Now $\id^{\leq p}(T) = \id(T) = 2$ and $\id^{\leq p}(I) = \id(I) = 1$. Thus, by Proposition~\ref{prop:H+I}, $\id^{\leq p}(G)\leq 3$, a contradiction.
If $v_1\neq w_1$ and $v_2\neq w_2$, then by the above arguments, $G$ has at most one edge $e$  incident to neither $v$ nor $w$. Thus $T$ is the union of the tree $T$ whose edges are those incident to $v$ or $w$, and 
the subgraph induced by the two endvertices of $e$.
Now $\id^{\leq p}(T) = \id(T) = 2$ and $\id^{\leq p}(I) = \id(I) = 1$. Thus, by Proposition~\ref{prop:H+I}, $\id^{\leq p}(G)\leq 3$, a contradiction.

\item If $v$ is the unique $3$-vertex of $G$, then one easily checks that $G$ is either  a tree of order at most $7$, or a triangle plus a pendant path of length at most $2$, or a 4-cycle and a pendant edge, or a 5-cycle with a pendant edge.
If $G$ is a tree, then $\id^{\leq p}(G) = \id(G) = 2$, a contradiction. If $G$ is a triangle plus a pendant edge, then $G$ is the union of a star of order $3$ and an edge which forms an induced $K_2$. Those two graphs have $(\leq p)$-inversion number $1$, thus, by Proposition~\ref{prop:H+I}, $\id^{\leq p}(G)\leq 2$, a contradiction. In all other cases, $|E(G)| \geq 5$ and $G$ is the union of a tree $T$ of order at most $6$ and an edge which forms an induced $K_2$.
Now $\id^{\leq p}(T) = \id(T) = 2$ and $\id^{\leq p}(K_2) = \id(K_2) = 1$. Thus, by Proposition~\ref{prop:H+I}, $\id^{\leq p}(G)\leq 3  \leq \left\lceil\frac{|E(G)|}{4}\right \rceil +1$, a contradiction.
\end{itemize}

\end{proof}

\subsection{Better bounds for not too sparse graphs}
%%%%%%%%%%%%%%%%%%%%%%%%%%%%%%%%%%%%%%%%%%%%%%%%

\begin{theorem}\label{thm:procedure1}
Let $p$ be an integer greater than $2$ and let $G$ be a graph. Then $$\id^{\leq p}(G) \leq \frac{1}{p-1}|E(G)| + \frac{p-2}{p-1}(|V(G)| -1).$$
\end{theorem}
\begin{proof}
We prove the result by induction on $|V(G)|$, the result holding trivially if $|V(G)|=1$.

Assume now that $|V(G)|\geq 1$.
Let $v$ be a vertex of $G$.
By Corollary~\ref{cor:G-v}~(i), 
$\id^{\leq p}(G) \leq \id^{\leq p}(G-v) + \left\lceil \frac{d(v)}{p-1}\right \rceil$.
Moreover, by the induction hypothesis, 
$\id^{\leq p}(G-v) \leq \frac{1}{p-1}|E(G-v)| + \frac{p-2}{p-1}(|V(G)| -2)$.
Thus 
\begin{eqnarray*}
\id^{\leq p}(G) & \leq &  \frac{1}{p-1}|E(G-v)| + \frac{p-2}{p-1}(|V(G)| -2) + \left\lceil \frac{d(v)}{p-1}\right \rceil\\
& = & \frac{1}{p-1}|E(G)| - \frac{d(v)}{p-1} + \left\lceil \frac{d(v)}{p-1}\right \rceil  + \frac{p-2}{p-1}(|V(G)| -2) \\
& \leq &  \frac{1}{p-1}|E(G)| + \frac{p-2}{p-1} + \frac{p-2}{p-1}(|V(G)| -2)\\
& = & \id^{\leq p}(G-v) + \left\lceil \frac{d(v)}{p-1}\right \rceil
\end{eqnarray*}
\end{proof}

\section{Connected graphs}\label{sec:upper-connected}
%%%%%%%%%%%%%%%%%%%%%%%%

We shall need the following result due to Kotzig~\cite{kotzig1957theory}.
\begin{lemma}[Kotzig~\cite{kotzig1957theory}]\label{lem:decomp-P3}
Every connected graph $G$ can be edge-decomposed into $\left\lceil \frac{|E(G)|}{2} \right\rceil$ paths of order at most $3$. 
\end{lemma}

A {\bf triangle-transversal} in a graph $G$ is a set $F$ of edges such that of $G\setminus F$ has no triangle. The {\bf triangle-transversal number} of a graph $G$, denotes by $\tau_3(G)$, is the minimum size of a triangle-transversal. Since every graph $G$ has a bipartite subgraph with at least $\left\lceil |E(G)|/2 \right \rceil$ edges, $\tau_3(G)\leq \left\lfloor |E(G)|/2 \right \rfloor$.

\begin{proposition}\label{prop:G-F-connected}
    Let $G$ be a connected graph.
    If $F$ is a minimum-size triangle-transversal of $G$, then $G\setminus F$ is connected.
\end{proposition}
\begin{proof}
Let $F$ be a minimum-size triangle-transversal of $G$.
Every edge $f$ of $F$ is contained in a triangle $T_f$ in $(G\setminus F)\cup f = G\setminus (F\setminus\{f\})$, for otherwise $F\setminus\{f\}$ would be a triangle-transversal of $G$.
Thus every walk in $G$ can be transformed in a walk in $G\setminus F$ with same end-vertices
by replacing every edge $f$ of $F$ by the path of length 2 between its end-vertices in $T_f$.
Thus as $G$ is connected, we get that $G\setminus F$ is also connected.
\end{proof}

\begin{lemma}\label{lem:E+tau}
    Let $G$ be a connected graph. Then $\displaystyle \id^{\leq 3}(G) \leq \left\lceil \frac{|E(G)| + \tau_3(G) }{2} \right\rceil$.
\end{lemma}

\begin{proof}
Let $G$ be a graph and let $F$ be a minimum triangle-transversal in $G$.

Let  $\vec{G_1}$ and $\vec{G_2}$ be two orientations of $G$. Let $\vec{H_1}$ and $\vec{H_2}$ be the orientations of $H=G\setminus F$ which are restrictions of $\vec{G_1}$ and $\vec{G_2}$ respectively.

Set $t =\lceil |E(H)|/2\rceil$.  By Proposition~\ref{prop:G-F-connected}, $H$ is connected, so by Lemma~\ref{lem:decomp-P3}, it can be edge-decomposed into $t$ paths $P_1, \dots, P_t$ of order at most $3$. 
For $i\in [t]$, there is a set $X_i\subset V(P_i)$ whose inversion transforms $\vec{H_1}\langle V(P_i) \rangle$ and $\vec{H_2}\langle
V(P_i) \rangle$.
Then, inverting $(X_i)_{i\in [t]}$, transforms $\vec{H_1}$ and $\vec{H_2}$.
Thus $\Inv(\vec{G_1}; (X_i)_{i\in [t]})$ and $\vec{G_2}$
disagree only on edges of $F$.
Each disagreeing edge can be reversed by inverting the set of its end-vertices.
Thus $\vec{G_1}$ can be transformed into $\vec{G_2}$
by inverting at most $t+|F|$ $(\leq 3)$-sets.

Therefore $\id^{\leq 3}(G) \leq t+|F| = \lceil |E(H)|/2\rceil + |F|= \left\lceil \frac{|E(G)|-|F|}{2}\right\rceil + |F| = \left\lceil \frac{|E(G)| + \tau_3(G) }{2} \right\rceil$. 
\end{proof}

The fact that $\tau_3(G)\leq \lfloor |E(G)|/2\rfloor$ and Lemma~\ref{lem:E+tau} immediately imply Theorem~\ref{thm:connected}.

\medskip

The {\bf triangle-edge-packing number} of a graph $G$, denoted by $\nu_3(G)$, is the maximum number of edge-disjoint triangles in $G$. 
Note that $\nu_3(G)\leq \tau_3(G)$ and
$\nu_3(G)\leq |E(G)|/3$.

\begin{lemma}\label{lem:P+nu}
    Let $G$ be a graph  and $\mathcal{P}$ be a decomposition of $E(G)$ into paths of order at most $3$. Then $\id^{\leq 3}(G)) \leq |\mathcal{P}| + \nu_3(G)$.
\end{lemma}

\begin{proof}We prove the result by induction on $k = |\mathcal{P}|$, the result holding clearly when  $k = 1$.

Assume $k > 1$. Let $\vec{G_1}$ and $\vec{G_2}$ be two orientations of $G$. Let $P$ be a path in $\mathcal{P}$ and let $G'$ be the subgraph obtained from $G$ by deleting $E(G\langle V(P)\rangle)$.
Let $\mathcal{P'}$ be the restriction of $\mathcal{P}$ in $G'$.
Let $\vec{G'_1}, \vec{G'_2}$ be the restrictions of $\vec{G_1}$ and $\vec{G_2}$ to $G'$, respectively. By the induction hypothesis, there is a family $\mathcal{X}$ of at most $k - 1 + \nu_3(G')$ $(\leq 3)$-sets such that $\Inv(\vec{G'_1};\mathcal{X}) =\vec{G'_2}$. Thus $\vec{G_3}=\Inv(\vec{G_1};\mathcal{X})$ and $\vec{G_2}$ disagree only on edges in $E(G\langle V(P)\rangle)$.

Suppose first that $G\langle V(P)\rangle$ is not a triangle. Then, we can transform $\vec{G_3}\langle V(P)\rangle$ into $\vec{G_2}\langle V(P)\rangle$ by inverting at most one subset of $V(P)$. Thus, there is family of at most $|\mathcal{X}| + 1 = k - \ell + \nu_3(G') + 1 \leq k + \nu_3(G)$ $(\leq 3)$-sets whose inversions transform $\vec{G_1}$ into $\vec{G_2}$ and the result holds.

Assume now that $G\langle V(P_1)\rangle$ is a triangle. Then $\nu_3(G) \geq \nu_3(G') + 1$.
Moreover,  we can transform $\vec{G_3}\langle V(P)\rangle$ into $\vec{G_2}\langle V(P)\rangle$ by inverting at most two subsets of $V(P)$.  Thus, there is family of at most $|\mathcal{X}| + 2$ $(\leq 3)$-sets whose inversions transform $\vec{G_1}$ into $\vec{G_2}$. But  $|\mathcal{X}| + 2 = k - 1 + \nu_3(G') + 2 \leq k + \nu_3(G)$, so the result holds.
\end{proof}

Lemma~\ref{lem:decomp-P3} and~\ref{lem:P+nu} immediately imply the following.

\begin{corollary}\label{cor:E+nu}
  If $G$ is a connected graph, then $\id^{\leq 3}(G)) \leq \left\lceil \frac{|E(G)|}{2} \right\rceil + \nu_3(G)$.
\end{corollary}

\section{Trees and forests}\label{sec:trees}
%%%%%%%%%%%%%%%%%%%%%%%%%%%%%%%%%%%%%%%%%%%%%%%%%%

The aim of this section is to prove Theorem~\ref{thm:trees}.
Recall that $\mathcal{F}$ denotes the class of the forests.
We shall need the following lemma.

\begin{lemma}\label{lem:t'andt}
    If there exist two constants $\alpha$ and $\beta$ such that $\conv^{\leq p}_{\cal F}(n) \leq \alpha n + \beta$ for all nonnegative integer $n$, then
    $\id^{\leq p}_{\cal F}(n) \leq \alpha n + 2\beta$ for all nonnegative integer $n$.    
\end{lemma}
\begin{proof}
    Let $F$ be a forest of order $n$, and let $\vec{F_1}$ and $\vec{F_2}$ be two orientations of $F$. 
    Consider the graph $H_{\neq}$ whose vertices are the connected components of $(V(F),E_{\neq})$ ,
    and in which two such connected components $C$ and $C'$ are adjacent if and only if
    there is an edge $uv$ in $E(F)$ with $u \in C$ and $v \in C'$.
    Observe that $H_{\neq}$ is a minor of $F$, and so is a forest. In particular,
    $H_{\neq}$ is bipartite. Let $(A'_1, A'_2)$ be a bipartition of $H_{\neq}$ and let $(A_1,A_2)$ the partition of $V(F)$ where a vertex is contained in $A_j$ if it is contained in a component of $H_{\neq}$ that belongs to $A_j'$ for $j\in [2]$.
    For every edge $uv$ in $E(F)$, we have $uv \in E_{\neq}$
    if and only if $\{u,v\} \subseteq A_1$ or $\{u,v\} \subseteq A_2$.
    For $j\in [2]$, let $F_j=F\langle A_j\rangle$, and let $n_j=|A_j|$.
    Since $\conv^{\leq p}_{\cal F}(n_j)\leq \alpha n_j + \beta$, there is a family $(X_i)_{i\in I_j}$ of at most  $\alpha n_j + \beta$ subsets of $A_j$ of size at most $p$ whose inversion reverses all edges of $F_j$.
    Observe that by the construction of $A_1$ and $A_2$, the inversion of an $X_i\subseteq A_j$ in $F$ only reverses edges with both end-vertices in $X_j$. 
    Thus $(X_i)_{i\in I_1\cup I_2}$ is family of at most $\alpha n_1 + \beta + \alpha n_2 + \beta = \alpha n + 2\beta$ sets 
    that transforms $\vec{F_1}$ into $\vec{F_2}$.
    Thus $\id^{\leq p}(F)\leq \alpha n + 2\beta$.
\end{proof}

\subsection{Case \texorpdfstring{$p=3$}{p=3}}
%%%%%%%%%%%%%%%%%%%%%%%%
The aim of this subsection is to prove assertion (i) of Theorem~\ref{thm:trees} which we restate.
\begin{theorem}\label{thm:tree2}
Let $T$ be a tree of order $n$.
Then $\conv^{\leq 3}(T)=\id^{\leq 3}(T) = \left\lceil \frac{n-1}{2}\right\rceil$.
\end{theorem}
\begin{proof}
Let $\vec{T}$ an orientation of $T$ and $\cev{T}$ its converse.
Then $|E_{\neq}| =|E(T)|=n-1$. 
Observe that a $(\leq 3)$-inversion reverses at most two edges of $T$, so $\vec{T}$ and $\cev{T}$ are at distance at least $\lceil (n-1)/2 \rceil$ in $\mathcal{I}^{\leq 3}(T)$.
Hence $\id^{\leq 3}(T) \geq \conv^{\leq 3}(T) \geq \left\lceil \frac{n-1}{2}\right\rceil$.

Now the tree $T$ has no triangle, so $\tau_3(T) = \nu_3(T) = 0$. Hence, by Lemma~\ref{lem:E+tau} or Corollary~\ref{cor:E+nu},
$\id^{\leq 3}(T) \leq \left\lceil \frac{|E(T)|}{2}\right\rceil  = \left\lceil \frac{n-1}{2}\right\rceil$
\end{proof}

\subsection{Case \texorpdfstring{$p=4$}{p=4}}
%%%%%%%%%%%%%%%%%%%%%%%%

The aim of this subsection is to prove the assertion (ii) of Theorem~\ref{thm:trees} which states  $\id^{\leq 4}_{\cal F}(n), \conv^{\leq 4}_{\cal F}(n)=\frac{3}{8}n + \Theta(1)$. The upper bound is given is Theorem~\ref{thm:tree4} and the lower bound in Proposition~\ref{prop:tree4-tight}.
In order to prove them, we need some preliminaries.

\begin{comment}
\medskip

\begin{lemma}\label{lem:indepenedent-pairs}
    Let $M$ be a matching in a tree $T$. Then $T$ contains at least 
    $\left\lceil \frac{|M|-2}{2}\right\rceil$ 
    disjoint pairs of independent edges.
\end{lemma}
\begin{proof}
    It follows directly from the fact that among three pairwise non-adjacent edges in a tree, there is a pair of independent edges.
\end{proof}
\end{comment}

\begin{lemma}\label{lem:decomp-tree4}
Every tree of order $n$ can be decomposed in $\lceil 3(n-1)/8 \rceil$ edge-disjoint induced subgraphs of order at most~$4$.
\end{lemma}
\begin{proof}
By induction on $n$, the result holding trivially when $n\leq 4$.

Let $T$ be a tree of order $n\geq 4$.
Let $(v_1, \dots, v_t)$ be a longest path in $T$.

If $d(v_2) \geq 4$, then let $w_1, w_2, w_3$ be three neighbours of $T$ distinct from $v_3$.
By the maximality of $(v_1, \dots, v_t)$, $T-\{w_1, w_2, w_3\}$ is connected.
Therefore, by the induction hypothesis, 
$T-\{w_1, w_2, w_3\}$ has a decomposition in $\lceil 3(n-4)/8 \rceil \leq \lceil 3(n-1)/8 \rceil - 1$ edge-disjoint induced subgraphs of order at most $4$,
which together with $T\langle \{w_1, w_2, w_3, v_2\}\rangle$ yields the result.

If $d(v_2) = 3$, then let $v_1, w_2, v_3$ be the three neighbours of $T$.
By the induction hypothesis,
$T-\{v_1, v_2, w_2\}$ has a decomposition in $\lceil 3(n-4)/8 \rceil$ edge-disjoint induced subgraphs of order at most $4$, 
which together with $T\langle \{v_1, v_2, w_2, v_3\}\rangle$ yields the result.

Now assume $d(v_2)=2$.

If $d(v_3)=2$, then by the induction hypothesis,
$T-\{v_1, v_2, v_3\}$ has a decomposition in $\lceil 3(n-4)/8 \rceil$ edge-disjoint induced subgraphs of order at most $4$, 
which together with $T\langle \{v_1, v_2, v_3, v_4\}\rangle$ yields the result.
Henceforth, we assume $d(v_3)\geq 3$.
Let $w_2$ be a neighbour of $v_3$ distinct from $v_2$ and $v_4$. 

By a similar reasoning as above, we have the result if $d(w_2) \geq 3$. Henceforth, we may assume $d(w_2) \leq 2$.

If $d(w_2)=1$, then
by the induction hypothesis, $T-\{v_1, v_2, w_2\}$ has a decomposition in $\lceil 3(n-4)/8 \rceil$ edge-disjoint induced subgraphs of order at most $4$, which together with $T\langle \{v_1, v_2, v_3, w_2\}\rangle$ yields the result.

Henceforth, we may assume that all neighbours of $v_3$ distinct from $v_4$ have degree $2$.
In particular, this implies that $t\geq 5$.

Assume $d(v_3)\geq 5$.
Let $v_2, w_2, x_2, y_2$ be neighbours of $v_3$ distinct from $v_4$.
Let $w_1$ (resp. $x_1$, $y_1$) be the neighbour of $w_2$ (resp. $x_2$, $y_2$) distinct from $v_3$.
By the induction hypothesis, $T-\{v_1, w_1, x_1, y_1, v_2, w_2, x_2, y_2\}$ has a decomposition in
$\lceil 3(n-9)/8 \rceil \leq \lceil 3(n-1)/8 \rceil - 3$ edge-disjoint induced subgraphs of order at most $4$,
which together with $T\langle \{v_1, v_2, v_3, w_2\}\rangle$, $T\langle \{x_1, x_2, v_3, y_2\}\rangle$, and  $T\langle \{w_1, w_2, y_1, y_2\}\rangle$ yields the result. 

Assume now that $d(v_3) = 4$.
Let $v_2, w_2, x_2$ be neighbours of $v_3$ distinct from $v_4$. Let $w_1$ (resp. $x_1$) be the neighbour of $w_2$ (resp. $x_2$) distinct from $v_3$.
By the induction hypothesis,
$T-\{v_1, w_1, x_1, v_2, w_2, x_2, v_3, v_t\}$ has a decomposition in $\lceil 3(n-9)/8 \rceil$ edge-disjoint induced subgraphs of order at most $4$,
which together with $T\langle \{v_1, v_2, v_3, w_2\}\rangle$, $T\langle \{x_1, x_2, v_3, v_4\}\rangle$, 
and $T\langle \{w_1, w_2, v_{t-1}, v_t\}\rangle$ yields the result. 

Henceforth, we assume $d(v_3) = 3$.
Let $w_2$ be its neighbour distinct from $v_2$ and $v_4$, and $w_1$ the neighbour of $w_2$ distinct from $v_3$.
By symmetry, we may assume that $d(v_{t-2})=3$, and the two neighbours of $v_{t-2}$ distinct from $v_{t-3}$, say $v_{t-1}$ and $w_{t-1}$, have degree $2$.
Let $w_t$ be the neighbour of $w_{t-1}$ distinct from $v_{t-2}$.
If $t=5$, then $V(T) = \{v_1,w_1,v_2,w_2,v_3,v_4,v_5\}$ and $E(T) = \{v_1v_2,w_1w_2,v_2v_3,w_2v_3,v_3v_4,v_4v_5\}$,
and so $T$ can be decomposed in $T\langle\{v_1,v_2,v_3,v_4\}\rangle$, $T\langle\{v_3,w_2\}\rangle$, and $T\langle\{w_1,w_2,v_4,v_5\}\rangle$ as wanted.
Now suppose $t \geq 6$.
In particular, $v_1,w_1,v_2,w_2,w_{t-1},w_t$ are distinct.
By the induction hypothesis, $T-\{v_1, w_1, v_2, w_2,v_{t}, w_t, v_{t-1}, w_{t-1}\}$ has a decomposition in
$\lceil 3(n-9)/8 \rceil$ edge-disjoint induced subgraphs of order at most $4$,
which together with $T\langle \{v_1, v_2, v_3, w_2\}\rangle$, $T\langle \{v_t, v_{t-1}, v_{t-2}, w_{t-1}\}\rangle$, 
and $T\langle \{w_1, w_2, w_{t-1}, w_t\}\rangle$ yields the result. 
\end{proof}

\begin{theorem}\label{thm:tree4}
Let $T$ be a tree of order $n$.
Then $\conv^{\leq 4}(T) \leq \left\lceil \frac{3(n-1)}{8}\right\rceil$.
\end{theorem}
\begin{proof}
Let $\vec{T}$ an orientation of $T$ and $\cev{T}$ its converse.
Set $s= \left\lceil \frac{3(n-1)}{8}\right\rceil$.
By Lemma~\ref{lem:decomp-tree4}, $T$ can be decomposed into $s$ induced subgraphs $S_1, \dots, S_s$ of order at most $4$.
Since the $S_i$ are edge-disjoint and induced, inverting some $V(S_{i_0})$ does not reverse any edge of the $S_i$ with $i\neq i_0$.
Thus inverting all $V(S_i)$ reverse every edge exactly once.
Thus $\Inv(\vec{T} ; (V(S_i))_{i\in [s]}) = \cev{T}$, and so $\conv^{\leq 4}(T) \leq s$.   
\end{proof}

\begin{corollary}
    $\id^{\leq 4}_{\cal F}(n) \leq \frac{3}{8} n + 1$.
\end{corollary}
\begin{proof}
For every nonnegative integer $n$, we have $\conv^{\leq 4}_{\cal F}(n) \leq \frac{3}{8} n + \frac{1}{2}$ by Theorem~\ref{thm:tree4}. 
Thus, by Lemma~\ref{lem:t'andt}, $\id^{\leq 4}_{\cal F}(n) \leq \frac{3}{8} n + 1$.
\end{proof}

This theorem is tight up as shown by the following proposition.
\begin{proposition}\label{prop:tree4-tight}
 For every positive integer $n$, there is a tree $T$ of order $n$ such that    $\conv^{\leq 4}(T) \geq \left\lceil \frac{3n-4}{8}\right\rceil$.
\end{proposition}
\begin{proof}
Let  $n$ be a positive integer. Set $s=\lfloor (n-1)/2\rfloor$ and $\epsilon = n-1 -2s$.
Let $T$ be the tree with vertex set $\{x\} \cup \{y_i \mid i\in [s+\epsilon]\} \cup \{z_i \mid i\in [s]\}$ and edge set $\{xy_i \mid i\in [s+\epsilon]\} \cup \{y_iz_i \mid i\in [s]\}$.
Let $\vec{T}$ be an orientation of $T$ and $\cev{T}$ its converse.
Then $|E_{\neq}| =|E(T)|$. 

Put a weight of $1$ on each
edge $xy_i$ and a weight of $2$ on each edge $y_iz_i$. See Figure~\ref{fig:tree-lb-p4}).
Then the sum of the weights of the edges is $3s + \epsilon$.
Observe that the sum of the weights of the edges reversed by a $(\leq 4)$-inversion is at most $4$.
Thus $\vec{T}$ and $\cev{T}$ are at distance at least $\lceil (3s +\epsilon)/4) \rceil \geq \lceil \frac{3n-4}{8}\rceil$
in $\mathcal{I}^{\leq 4}(T)$.
Hence $\conv^{\leq 4}(T) \geq \left\lceil \frac{3n-4}{8}\right\rceil$.
\end{proof}

\begin{figure}[hbtp]
    \centering
    \includegraphics{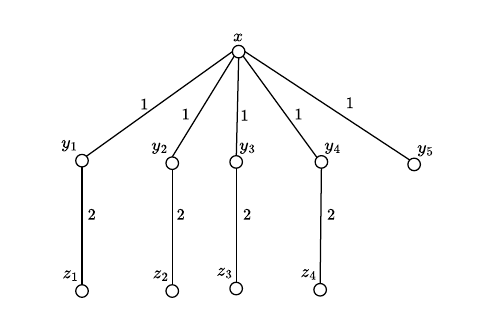}
    \caption{Example of the tree $T$ defined in Proposition~\ref{prop:tree4-tight}, when $n = 10$.}
    \label{fig:tree-lb-p4}
\end{figure}

\subsection{Case \texorpdfstring{$p=5$}{p=5}}

%%%%%%%%%%%%%%%%%%%%%%%%

The aim of this subsection is to prove the assertion (iii) of Theorem~\ref{thm:trees} which states  $\id^{\leq 5}_{\cal F}(n), \conv^{\leq 5}_{\cal F}(n)=\frac{2}{7}n + \Theta(1)$. The upper bound is given is Theorem~\ref{thm:tree5} and the lower bound in Proposition~\ref{prop:tree5-tight}.
In order to prove them, we need some preliminaries.

\medskip
 
Let $T$ be a tree of order $n$.
  A {\bf good $5$-set with root $t$} in $T$ is a set $X$ of five vertices such that $T\langle X\rangle$ is a tree and $T-(X\setminus\{t\})$ is a tree.

A {\bf good $5$-pair with roots $t_1,t_2$} in a tree $T$ is a pair $(X_1, X_2)$ of sets of at most five vertices such that $T\langle X_1\rangle$ has four edges, $T\langle X_2 \rangle$ has three edges, $T\langle X_1 \cap X_2 \rangle$ has no edge, and 
$T - ((X_1 \cup X_2) \setminus \{t_1, t_2\})$ is a tree.
%$(T-(X_1\setminus\{t_1\}))\cup (X_2\setminus\{t_2\})$ is a tree.

\begin{lemma}\label{lem:good5}
Every tree of order $n\geq 5$ contains either either a good $5$-set or a good $5$-pair. 
\end{lemma}
\begin{proof}
We prove the result by considering a minimum counterexample $T$.
Clearly, $n \geq 6$.

\begin{claim}\label{claim:3leaves}
    Every vertex of $T$ is adjacent to at most three leaves.
\end{claim}
\begin{subproof}
If a vertex $v$ is adjacent to four leaves $u_1, u_2, u_3, u_4$, then $\{u_1 ,u_2, u_3, u_4, v\}$ is a good $5$-set with root $v$, a contradiction.
\end{subproof}

Let $(v_1, \dots , v_{\ell})$ be a longest path in $T$.

\begin{claim}\label{claim:v2}
$d(v_2) = 2$
\end{claim}
\begin{subproof}
By Claim~\ref{claim:3leaves},  $d(v_2) \leq 4$. 

If $d(v_2) = 4$, then let $v_1, w_1, x_1, v_3$ be the four neighbours of $v_2$. The set $\{v_1, w_1, x_1,  v_2, v_3\}$ is a good $5$-set with root $v_3$, a contradiction.
Henceforth, we may assume $d(v_2) = 3$.
 Let $u_1$ be the neighbour of $v_2$ distinct from $v_1$ and $v_3$.

If $d(v_3) = 2$, then $\{u_1, v_1, v_2, v_3, v_4\}$ is a good $5$-set with root $v_4$, a contradiction.
Henceforth we may assume $d(v_3) \geq 3$. Let $w_2$ be a neighbour of $v_3$ distinct from $v_2$ and $v_4$.

If $d(w_2)=1$, then  $\{u_1, v_1, v_2, w_2, v_3\}$ is a good $5$-set with root $v_3$, a contradiction.
Henceforth we may assume that $d(w_2)\geq 2$. In particular, this implies that $\ell \geq 5$.

Let $w_1$ be a vertex adjacent to $w_2$ distinct from $v_3$.
Then $(w_1, w_2, v_3, \dots , v_{\ell})$ is a longest path in $T$. Thus by Claim~\ref{claim:3leaves}, and the above argument,
we have $d(w_2)\leq 3$.

Assume for a contradiction that $d(w_2)=2$. 
If $d(v_{\ell-1})\geq 3$, let $w_\ell$ be a neighbour of $v_{\ell-1}$ distinct from $v_{\ell-2}$. 
Then $(\{u_1, v_1, v_2, w_2, v_3\}, \{w_1, w_2, v_\ell, w_\ell, v_{\ell-1}\})$ is a good $5$-pair with roots $v_3, v_{\ell-1}$.
If $d(v_{\ell-1}) = 2$, then $(\{u_1, v_1, v_2, w_2, v_3\}, \{w_1, w_2, v_\ell, v_{\ell-1}, v_{\ell-2}\})$ is a good $5$-pair with roots $v_3, v_{\ell-2}$. In both cases, we get a contradiction, so $d(w_2)=3$. Let $w_1$ and $x_1$ be the neighbours of $w_2$ distinct from $v_3$. 
Then $(\{u_1, v_1, v_2, w_2, v_3\}, \{w_1, x_1, w_2, v_{\ell}, v_{\ell-1}\})$ is a good $5$-pair with roots $v_3, v_{\ell-1}$, a contradiction.
\end{subproof}

\begin{claim}\label{claim:v3}
$d(v_3) = 2$.
\end{claim}
\begin{subproof}
If a neighbour of $v_3$ distinct from $v_4$ has degree at least $2$, the it is the second vertex of a longest path, and so by Claim~\ref{claim:v2} it has degree at most $2$.

Suppose $v_3$ has a neighbour $w_2$ distinct from $v_2$ and $v_4$ of degree $2$. Let $w_1$ be its neighbour distinct from $v_3$. Then $\{v_1, w_1, v_2, w_2, v_3\}$ is a good $5$-set with root $v_3$, a contradiction.
Henceforth all neighbours of $v_3$ distinct from $v_2$ and $v_4$ are leaves. If $d(v_3)\geq 4$, let $w_2$ and $x_2$ be two neighbours of $v_3$ distinct from $v_2$ and $v_4$.
Then $\{v_1, v_2, w_2, x_2, v_3\}$ is a good $5$-set with root $v_3$, a contradiction.
If $d(v_3) = 3$, let $w_2$ be the neighbour of $v_3$ distinct from $v_2$ and $v_4$.
Then $\{v_1, v_2, w_2, v_3, v_4\}$ is a good $5$-set with root $v_4$, a contradiction.
Thus $d(v_3) = 2$.
\end{subproof}

Now $d(v_4)\geq 3$ for otherwise $\{v_1, v_2, v_3, v_4, v_5\}$ would be a good $5$-set with root $v_5$.

Let $w_3$ be a neighbour of $v_4$ distinct from $v_3$ and $v_5$.

Note that $w_3$ is not a leaf for otherwise $\{v_1, v_2, v_3, w_3, v_4\}$ would be a good $5$-set with root $v_4$.
In particular, $\ell \geq 6$.

If there is a path $(w_1, w_2, w_3)$ in $T\setminus w_3v_4$, then $(w_1, w_2, w_3, v_4 \dots , v_{\ell})$ is a longest path in $T$.
Thus, by Claims~\ref{claim:v2} and~\ref{claim:v3}, $d(w_2)=d(w_3)=2$, and then $(\{v_1, v_2, v_3, w_3, v_4\},$ $\{w_1, w_2, w_3, v_{\ell}, v_{\ell-1}\})$ is a good $5$-pair with roots $v_4, v_{\ell-1}$, a contradiction.
Henceforth, all neighbours of $w_3$ distinct from $v_4$ is a leaf. Let $W_2$ be the set of leaves adjacent to $w_3$.
By Claim~\ref{claim:3leaves}, $|W_2|\leq 4$, and $|W_2|\geq 1$ since $w_3$ is not a leaf.

If $|W_2| = 3$, then $(\{v_1, v_2, v_3, w_3, v_4\}, W_2\cup \{w_3\})$ is a good $5$-pair with roots $v_4, w_3$, a contradiction.
If $|W_2| = 2$, then $(\{v_1, v_2, v_3, w_3, v_4\},$ $\{w_1, w_2, w_3, v_{\ell}, v_{\ell-1}\})$ is a good $5$-pair with roots $v_4, v_{\ell-1}$, a contradiction.

Hence, $w_3$ is a adjacent to a unique leaf $w_2$.
Similarly to Claim~\ref{claim:v2}, $d(v_{\ell-1}) = 2$. Then $(\{v_1, v_2, v_3, w_3, v_4\}, \{w_2, w_3, v_\ell, v_{\ell-1}, v_{\ell-2}\})$ is a good $5$-pair with roots $v_4, v_{\ell-2}$, a contradiction.

This completes the proof of Lemma~\ref{lem:good5}.
\end{proof}

\begin{theorem}\label{thm:tree5}
If $T$ is a tree of order $n$, then $\conv^{\leq 5}(T) \leq \left\lceil \frac{2n-2}{7}\right\rceil$.
\end{theorem}
\begin{proof}
We prove the result by induction on $n$, the result holding trivially if $n\leq 4$. 

Let $T$ be a tree of order $n\geq 5$.
By Lemma~\ref{lem:good5}, $T$ has either a good $5$-set or a good $5$-pair.

Assume first that $T$  has a  good $5$-set $X$ with root $t$.  By the induction hypothesis, we can reverse all the arcs of $T-(X\setminus\{t\})$ in at most $\left\lceil \frac{2n-10}{7}\right\rceil$ $(\leq 5)$-inversions.
Inverting next $X$, we have all the arcs reversed.
Thus $\conv^{\leq 5}(T) \leq\left\lceil \frac{2n-10}{7}\right\rceil + 1 <  \left\lceil \frac{2n-2}{7}\right\rceil$. 

Assume now that $T$  has a good $5$-pair $(X_1, X_2)$ with roots $t_1,t_2$.  By the induction hypothesis, we can reverse all the arcs of $T-((X_1\cup X_2)\setminus\{t_1,t_2\})$ in at most $\left\lceil \frac{2n-16}{7}\right\rceil$ $(\leq 5)$-inversions.
Inverting next $X_1$ and $X_2$, we have all the arcs reversed.
Thus $\conv^{\leq 5}(T) \leq \left\lceil \frac{2n-16}{7}\right\rceil +2 = \left\lceil \frac{2n-2}{7}\right\rceil$.

In both cases,  $\conv^{\leq 5}(T) \leq \left\lceil \frac{2n-2}{7}\right\rceil$. 
\end{proof}

Theorem~\ref{thm:tree5} and Lemma~\ref{lem:t'andt} yield the following.

\begin{corollary}
    $\id^{\leq 5}_{\cal F}(n) \leq \frac{2}{7} n + \frac{8}{7}$.
\end{corollary}
\begin{proof}
For every nonnegative integer $n$, we have $\conv^{\leq 5}_{\cal F}(n) \leq \frac{2}{7} n + \frac{4}{7}$ by Theorem~\ref{thm:tree5}. 
Thus, by Lemma~\ref{lem:t'andt}, $\id^{\leq 5}_{\cal F}(n) \leq \frac{2}{7} n + \frac{8}{7}$.
\end{proof}

The previous two results are tight up to a small additive constant as shown by the following
proposition.

\begin{proposition}\label{prop:tree5-tight}
 For every positive integer $n$, there is a tree $T$ of order $n$ such that $\conv^{\leq 5}(T) \geq 2 \left\lfloor \frac{n-1}{7}\right\rfloor$.
\end{proposition}
\begin{proof}
It suffices to prove the result when $n-1 \equiv 0 \mod 7$.
So let $n = 7q+1$. Let $T$ be the tree defined by
\begin{eqnarray*}
    V(T) & = & \{r\}\cup \bigcup_{j=1}^q \{x_j, y_j^1, y_j^2, z_j^1, z_j^2, z_j^3, z_j^4 \} ,   \\
     E(T) & = & \bigcup_{j=1}^q \{rx_j, x_jy_j^1, x_jy_j^2, y_j^1z_j^1, y_j^1z_j^2, y_j^2z_j^3, y_j^2z_j^4 \} .
\end{eqnarray*}

\begin{figure}[hbtp]
    \centering
    \includegraphics{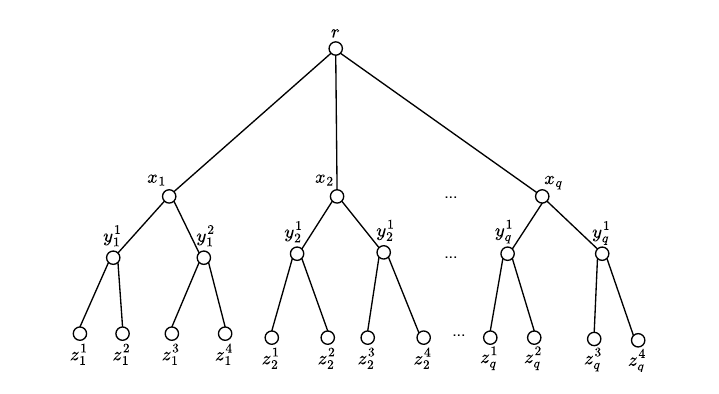}
    \caption{Tree $T$ defined in Proposition~\ref{prop:tree5-tight}.}
    \label{fig:tree-lb-p5}
\end{figure}

See Figure~\ref{fig:tree-lb-p5}.
For $j\in [q]$, the {\it $j$th branch} is the subtree $B_j$
with vertex set $\{r\}\cup \{x_j, y_j^1, y_j^2, z_j^1, z_j^2, z_j^3, z_j^4 \} $ and edge set $\{rx_j, x_jy_j^1, x_jy_j^2, y_j^1z_j^1, y_j^1z_j^2, y_j^2z_j^3, y_j^2z_j^4\} $.

Let $\vec{T}$ be an orientation of $T$ and let $\cev{T}$ be its converse, and let $(X_i)_{i\in I}$ be a family of sets whose inversions transform  $\vec{T}$ into $\cev{T}$.
The {\it trace} $T_{i,j}$ of a set $X_i$ on branch $B_j$ is the set of edges of $B_j$ with both end-vertices in $X_i$.
We assign weights to the traces as follows.
\begin{itemize}
    \item If $|T_{i,j}| = 0$, then $w(T_{i,j}) =0$.
    
    \item If $|T_{i,j}| = 1$, then $\displaystyle w(T_{i,j}) = \left\{
    \begin{array}{ll}
      5/3,&\text{if $rw_j\notin T_{i,j}$,}\\
    
      5/4, &\text{if $rw_j\in T_{i,j}$.}
    \end{array}\right.$
    
    \item If $|T_{i,j}| = 2$, then $\displaystyle w(T_{i,j}) = \left\{
    \begin{array}{ll}
      10/3,&\text{if $rw_j\notin T_{i,j}$,}\\
    
      9/4, &\text{if $rw_j\in T_{i,j}$.}
    \end{array}\right.$

    \item If $|T_{i,j}| = 3$, then $\displaystyle w(T_{i,j}) = \left\{
    \begin{array}{ll}
      5,&\text{if $rw_j\notin T_{i,j}$,}\\
    
      7/2, &\text{if $rw_j\in T_{i,j}$.}
    \end{array}\right.$

    \item If $|T_{i,j}| = 4$, then $w(T_{i,j}) =5$.
    \end{itemize}

On the one hand, the weight of each $X_i$ which is
$w(X_i) = \sum_{j\in [q]} w(T_{i,j})$ is at most $5$.

On the other hand, let us consider the weight of branch $B_j$ which is $w(B_j) = \sum_{i\in I} w(T_{i,j})$.
Observe that $B_j$ has at most one trace with four edges and that if it has one such trace, then $rw$ is not in a trace of size $2$ or $3$. Hence, one easily sees that $w(B_j)\geq 10$.

Thus $ 5|I| \geq  \sum_{i\in I} w(X_i)  = \sum_{j\in [q]} w(B_j) \geq 10 q$, so $|I| \geq 2q = 2 \left\lfloor \frac{n-1}{7}\right\rfloor$.
\end{proof}

\subsection{General case}
%%%%%%%%%%%%%%%%%%%%%%%%

The aim of this subsection is to prove the assertion (iv) of Theorem~\ref{thm:trees} which states  $\conv^{\leq p}_{\cal F}(n)  \leq \id^{\leq p}_{\cal F}(n) \leq \frac{n-1}{p- c\sqrt{p}} + 2$ with $c = \sqrt{2 + \sqrt{2}}$ for any $p\geq 4$.
We need the following lemma.

\begin{lemma}\label{lem:set-tree-sqrtp}
    Let $p \geq 4$ be an integer. Let $T$ be a tree with at least $p$ vertices rooted at a vertex $r$.  Then, there exists a set $X \subseteq V(T)$ such that
    \begin{enumerate}[label=(\roman*)]
        \item $|X| \leq p$,
        \item $T\langle X\rangle$ has at least $p - \sqrt{(2 + \sqrt{2})p}$ edges, and
        \item every non-trivial connected component in $T \setminus E(T\langle X\rangle)$ contains $r$. In particular, this implies that there is at most one non-trivial connected component in $T \setminus E(T\langle X\rangle)$.
    \end{enumerate}

\end{lemma}

\begin{proof}
    The proof follows by induction on $p + |V(T)|$. If $|V(T)| = p$, then $X = V(T)$ is a set satisfying Properties (i) to (iii). We may thus assume that $|V(T)| > p$.
    
     Let $r_1, \ldots, r_k$ be the neighbours of $r$ in $T$ and let $T_i$ be the subtree rooted at $r_i$. Without loss of generality, we may assume that $|V(T_1)| \geq \ldots \geq |V(T_k)|$.
    If $|V(T) \setminus V(T_k)| \geq p$, then by applying the induction to $p$, $T - V(T_k)$ and $r$, there is a set $X$ that satisfies Properties (i) to (iii) in $T - V(T_k)$ which directly implies that $X$ satisfies the Properties (i) to (iii) in $T$.
    
    Henceforth we may assume $|V(T) \setminus V(T_i)| < p$ for every $1 \leq i \leq k$. So $k \geq 2$. Let $\alpha = \sum_{i = 1}^{k - 1}|V(T_i)|$. Note that $p - \alpha > 0$ and $\alpha \geq \frac{|V(T)| - 1}{2} \geq \frac{p}{2}$. %\CR{Isn't it $\alpha \geq \frac{|V(T)|-1}{2}$, since $r$ is missing?} \carol{Yes,  fixed.}.
    Set $c = \sqrt{2 + \sqrt{2}}$.
    
    Suppose first that $|V(T_k)| \leq c\sqrt{p}$. Let $X = V(T) \setminus V(T_k)$. Note that $X$ clearly satisfies Properties (i) and (iii). Moreover, since $T\langle X\rangle$ is connected and has $\alpha + 1$ vertices, $|E(T\langle X\rangle)| = \alpha$. Thus,

    $$|E(T\langle X\rangle)| = \alpha \geq p - |V(T_k)| \geq p - c\sqrt{p},$$
    and $X$ also satisfies Property (ii).
   
    Henceforth, we may assume $|V(T_k)| > c\sqrt{p}$. 
    Since $|V(T_k)| \leq \frac{|V(T)|-1}{k}$, this implies that $p > (k - 1){c\sqrt{p}}$. We will now define a set $X' \subseteq V(T_k)$ satisfying the following properties.

    \begin{enumerate}[label=(\alph*)]
        \item $|X'| \leq p - \alpha$,
        \item $T_k\langle X'\rangle$ has at least $p - \alpha - c\sqrt{p - \alpha}$ edges, and
        \item every non-trivial connected component in $T_k \setminus E(T_k\langle X'\rangle)$ contains $r_k$.
    \end{enumerate}
    
    If $p - \alpha \leq 3$, then set $X' = \{r_k\}$. It is easy to check that Properties (a) to~(c) hold since $T_k\langle X'\rangle$ has no edges. Otherwise,  
    by the induction hypothesis applied to $p - \alpha$ 
    $T_k$, and $r_k$, there exists $X' \subseteq V(T_k)$ satisfying Properties~(a) to~(c).
    
    Let $X = X' \cup 
    (\bigcup_{i=1}^{k-1} V(T_i))$. 
    Note that $|X| = |X'| + \alpha \leq p - \alpha + \alpha \leq p$. So Property (i) is satisfied. Moreover, since $r$ and $r_k$ are adjacent, Property (iii) is also satisfied.
    It suffices to show Property (ii). Note that 
    \begin{align*}
         |E(T\langle X\rangle)| &\geq |E(T_k\langle X'\rangle)| + E(T\langle \textstyle\bigcup_{i=1}^{k-1} V(T_i)\rangle)| \\ 
         &\geq p - \alpha - c\sqrt{p - \alpha} + \alpha - (k - 1) \\
         &= p - c\sqrt{p - \alpha} - (k - 1)
    \end{align*}
Since  $\alpha \geq \frac{p}{2}$ and $p > (k - 1)c\sqrt{p}$, it follows that
\begin{align*}
        |E(T\langle X\rangle)| &\geq p - c\sqrt{\frac{p}{2}}- \frac{p}{c\sqrt{p}}\\
        &= p - \frac{c\sqrt{2p}}{2} - \frac{\sqrt{p}}{c}\\
        &= p - c\sqrt{p}\left(\frac{\sqrt{2}}{2} + \frac{1}{c^2}\right)
\end{align*}
Since $c = \sqrt{2 + \sqrt{2}}$, it follows that 
\begin{align*}
\frac{\sqrt{2}}{2} + \frac{1}{c^2} &= \frac{\sqrt{2}}{2} + \frac{1}{2 + \sqrt{2}} \\
    &= \frac{\sqrt{2}}{2} + \frac{2 - \sqrt{2}}{2}\\
    &= 1
\end{align*}
Thus, $|E(T\langle X\rangle)| \geq p - c\sqrt{p}$. This finishes the proof. 
\end{proof}

\begin{theorem}\label{thm:conv-tree}
    Let $p \geq 4$ be an integer. $\conv^{\leq p}_{\cal F}(n) \leq \left\lceil\frac{n-1}{p- c\sqrt{p}}\right\rceil$ with $c = \sqrt{2 + \sqrt{2}}$.
\end{theorem}
\begin{proof}
We prove the result by induction on $n$, the result holding trivially if $n\leq p$.

Assume now that $n\geq p$, and let $T$ be a tree on $n$ vertices.
By Lemma—\ref{lem:set-tree-sqrtp}, there exists a subset $X$ of $V(T)$ which satisfies the properties (i)-(iii) of this lemma.
Property (i) asserts that $|X|\leq p$, so it can be inverted.
Now $\Inv(T;X)$ disagrees with the converse of $T$ on a $E(T)\setminus E(T\langle X\rangle)$. By Property (iii), this set of edges induces a forest with unique connected component $T'$.
Hence $\conv^{\leq p}(T) \leq 1 + \conv^{\leq p}(T')$.
But by Property~(ii), $|E(T')| \leq |E(T)| - (p- c\sqrt{p})$.
Thus, by the induction hypothesis, $\conv^{\leq p}(T')\leq \left\lceil\frac{n-1 - (p- c\sqrt{p}}{p- c\sqrt{p}}\right\rceil = \left\lceil\frac{n-1}{p- c\sqrt{p}}\right\rceil -1$.
Hence $\conv^{\leq p}(T) \leq \left\lceil\frac{n-1}{p- c\sqrt{p}}\right\rceil$.
\end{proof}

This theorem and Lemma~\ref{lem:t'andt} directly imply the
assertion (iv) of Theorem~\ref{thm:trees}.

\section{\texorpdfstring{$k$}{k}-degenerate graphs}\label{sec:k-deg}
%%%%%%%%%%%%%%%%%%%%%%%%%%%%%%%%%%%%%%%%%%%%%%%%%%%%%%%%%%%%%%%%%%

Let $G$ be a graph.
It is {\bf $k$-degenerate} if it has a {\bf $k$-degenerate ordering}, that is, an ordering  $(v_1, \dots ,v_n)$ of $V(G)$ such that
$v_i$ has at most $k$ neighbours in $\{v_{i+1}, \dots , v_n\}$ for all $i\in [n-1]$.

\begin{lemma}\label{lem:vc}
Let $G$ be a $k$-degenerate graph and let $p \geq k + 1$ be an integer. 
Let $(v_1, \ldots, v_n)$ be a $k$-degenerate ordering of $V(G)$. If there exists $\ell \in [n]$ such that
$\{v_1, \ldots, v_\ell\}$ is a vertex-cover of $G$, then $\id^{\leq p}(G) \leq \ell$.
\end{lemma}

\begin{proof}
The results follows directly from an easy induction using Corollary~\ref{cor:G-v} and the fact that the $(\leq p)$-inversion diameter of an edgeless is $0$.     
\end{proof}

\begin{corollary}\label{cor:procedure-greedy}
Let $G$ be a $k$-degenerate graph. 
For any $p\geq k+1$, $\id^{\leq p}(G) \leq |V(G)| -1$.
\end{corollary}

\begin{proposition}
    For any positive integer $n$, there exists a $2$-degenerate graph $G^2_n$ of order $n$ such that $\id^{\leq 3}(G^2_n) = n -1$.
    Moreover if $n$ is even $\conv^{\leq 3}(G^2_n) = n -1$, and if $n$ is odd $\conv^{\leq 3}(G^2_n) = n -2$.
\end{proposition}
\begin{proof}

For $n \geq 2$, let $G^2_n$ be the graph obtained from $K_2$ by adding a set $X$ of $n-2$ vertices and all the edges between $X$ and $V(K_2)=\{u_1,u_2\}$. It is clearly $2$-degenerate.

Consider two orientations $\vec{G_1}$ and $\vec{G_2}$ of $G^2_n$ that disagrees on all edges of $G$ except on the edge of $u_1u_2$ when $n$ is odd.
There are $2(n-2)$ edges between $X$ and $V(K_2)$ and a $(\leq 3)$-inversion reverses at most two of them. Hence at least $n-2$ inversions are required to transform  $\vec{G}_1$ into $\vec{G}_2$.

Assume for a contradiction that there is a family of $n-2$ $(\leq 3)$-sets whose inversions transform  $\vec{G_1}$ into $\vec{G_2}$.
Then the inversion of each of those sets must reverse exactly two edges between $X$ and $K_2$ which are not reversed by the other sets. Thus the sets must be of the form $\{u_1, u_2, x\}$ for $x\in X$, or $\{u_j, x, x'\}$ for $j\in [2]$ and $x, x'\in X$.
There are two kinds of vertices $x$ of $X$: those of $X_1$ who belongs to one set of the form $\{u_1, u_2, x\}$, and those of $X_2$ who belongs to a set of the form $\{u_1, x, x'\}$ and one of the form $\{u_1, x, x"\}$.
Note that they is an even number of vertices in $X_2$ as each set of the from $\{u_1, x, x'\}$ contains two of them.
Hence $|X_1|$ has the same parity as $|X|$ and so as $n$.
So, if $n$ is even then the edge $u_1u_2$ is not reversed, and if $n$ is odd then the edge $u_1u_2$ is reversed. In both cases, its orientation disagrees with $\vec{G}_2$, a contradiction.

Hence $\vec{G_1}$ and $\vec{G_2}$ are at distance at least $n-1$ in $\mathcal{I}^{\leq 3}(G^2_n)$. This gives the results.
\end{proof}

\section{\texorpdfstring{$(\leq p)$}{(<=p)}-inversion diameter of planar graphs}\label{sec:planar}
%%%%%%%%%%%%%%%%%%%%%%%%%%%%%%%%%%%%%%%%%%%%%%%%%%

\subsection{Upper bounds on \texorpdfstring{$\id^{\leq p}_{\cal P}(n)$}{id(<= p)\_P(n)}}
%%%%%%%%%%%%%%%%%%%%%%%%%%%%%%%%%%%%%%%%%%%%%%%%%%%%%%%%

A planar graph of order $n$ has at most $3n-6$ edges.
Thus, by Theorem~\ref{thm:uppergen},
$\id^{\leq p}_{\cal P}(n) \leq \left\lceil\frac{3n-6}{\lfloor p/2\rfloor}\right \rceil + \frac{1}{2}p^2$. 
For small values of $p$, since every planar graph is $5$-degenerate, better upper bounds are given by Equation~\eqref{eq:m/p-1}: 
$\id^{\leq 3}_{\cal P}(n) \leq 2n - \frac{7}{2}$, $\id^{\leq 4}_{\cal P}(n) \leq \frac{5}{3}n - \frac{8}{3}$, and $\id^{\leq 5}_{\cal P}(n) \leq \frac{3}{2}n - \frac{9}{4}$.

In this section, we first slightly improve on the general upper bound given by Theorem~\ref{thm:uppergen}.  We then improve on the upper bounds on $\id^{\leq p}_{\cal P}(n)$ for $p=3,4,5$.

\medskip

It is known that if a planar graph contains a matching of size $k$, then it contains an induced matching of size $k/4$ (see~\cite{kanj2011induced}).

\begin{theorem}
  Let $p \geq 2$ be an integer and let $G$ be a planar graph. Then, $\id^{\leq p}(G) \leq \left\lceil\frac{|E(G)|}{\lfloor p/2\rfloor}\right \rceil + 8\lfloor p/2\rfloor - 8$. 
\end{theorem}
\begin{proof}
    Let $q = \lfloor p/2\rfloor$.
    Let $\vec{G}_1$ and $\vec{G}_2$ be two orientations of $G$.
    Let $G$ be a minimum counterexample to the statement and let $M$ be a maximum matching of $G$. By Lemma~\ref{lem:min-ce}, $\Delta(G)\leq q-1$ and $G$ has no induced matching of size $q$. Thus $|M| \leq 4(q - 1)$. Let $A$ be the set of vertices covered by $M$. Then $A$ is a vertex-cover of $G$ and $|A| \leq 8(q-1)$. 
    Let $(v_1, \ldots, v_n)$ be an ordering of $V(G)$ such that $\{v_1, \ldots, v_{|A|}\} = A$. Since $\Delta(G) \leq q - 1$, $v_1, \ldots, v_n$ is a $(q-1)$-degenerate ordering of $G$. By Lemma~\ref{lem:vc}, $\id^{\leq p}(G) \leq 8q - 8$, a contradiction.
\end{proof}

\begin{corollary}
    $\id^{\leq p}_{\cal P}(n) \leq \left\lceil\frac{3n-6}{\lfloor p/2\rfloor}\right \rceil + 8\lfloor p/2\rfloor - 8$
\end{corollary}

\begin{theorem}
If $G$ is a planar graph of order $n$, then 
$\id^{\leq 3}(G) \leq \frac{11}{6}n - \frac{8}{3}$ and $\id^{\leq 5}(G) \leq \id^{\leq 4}(G) \leq \frac{4}{3}n + \frac{10}{3}$
\end{theorem}
\begin{proof}
Let $G$ be a planar graph and let $p\in \{3,4\}$. 
Let $\vec{G}_1$ and $\vec{G}_2$ be two orientations of $G$.
We apply the following procedure.
\begin{enumerate}

\item First, as long as there is a $K_3$ with its three edges in $E_{\neq}$, we invert its vertex set. This reverses its three edges (which are removed from $E_{\neq}$).

 \item As long as there is a $4$-cycle $(v_1,v_2,v_3, v_4, v_1)$ with all its edges in $E_{\neq}$, we invert the sets $\{v_1, v_2, v_3\}$ and $\{v_3, v_4, v_1\}$.
 This reverses the four edges of the $4$-cycle (which are removed from $E_{\neq}$) and no other.

 \item As long as there are two edges
 $xy, yz \in E_{\neq}$ such that $xz\notin E(G)$, then we invert $\{x,y,z\}$.  This reverses the two edges $xy, yz$ (which are removed from $E_{\neq}$) without adding any new edges in $E_{\neq}$.
 \item [3+.] If $p=4$, then as long as there are two edges
 $wx, yz \in E_{\neq}$ such that $wy, wz, xy, xz \notin E(G)$, then we invert $\{w,x,y,z\}$.  This reverses the two edges $wx, yz$ (which are removed from $E_{\neq}$) without adding any new edges in $E_{\neq}$.

 \item Finally, we reverse the remaining edges of $E_{\neq}$ one by one.
\end{enumerate}

At Step 1, three edges are reversed per inversions; at Step 2, 3, and 3+, two edges (in average) are reversed per inversions; at Step 4, one edge is reversed per inversion.
For $i\in [4]$, let $E_i$ be the set of edges reversed at Step $i$, and set $m_i=|E_i|$.

The number of inversions of our procedure is
$N=m_1/3 + m_2/2 +m_3/2 + m_4$.

We have $m_1+m_2+m_3 + m_4 = |E_{\neq}| \leq |E(G)| \leq 3n-6$, because $G$ is planar.

Observe that after Step 1, the graph $(V(G),E_2\cup E_3 \cup E_4)$ has no triangle and is planar.
So it has at most $2n-4$ edges. Hence $m_2+m_3+m_4\leq 2n-4$.

Consider now the graph $H=(V(G),E_4)$.
\begin{itemize}
    \item It has no triangle, nor $4$-cycle, because of Step 1 and 2.
    \item The closed neighbourhood in $H$ of every vertex is a clique in $G$ for otherwise Step~3 would apply. Thus, as $G$ is planar and thus has no clique of size $5$, we get that $\Delta(H)\leq 3$.  
    \item It has no odd cycle, since every odd cycle of a planar graph contains two consecutive edges $xy, yz$ such that $xz\notin E(G)$, which should have been reversed at Step 3. 
    
    \item Let $C=(u_1, v_1, u_2, v_2, \dots , u_k, v_k, u_1)$ be an even cycle.
    Because Step 3 did not apply, $(u_1, u_2, \dots , u_k, u_1)$ and $(v_1, v_2, \dots , v_k, v_1)$ are cycles in $G$. Moreover, since Step~1 did not apply, no edge of those cycles is in $E_2\cup E_3\cup E_{4}$. Without loss of generality, the cycle
    $C'=(u_1, u_2, \dots , u_k, u_1)$ is outside $C$. Note that, in $H$, there is no path from $C$ to a vertex outside $C'$. Indeed, such a path would go through an edge $u_iw$ with $u_i\in V(C')$ and $w$ outside $C'$. But then 
    $u_iv_i$, $u_iw$ are edges in $E_4$ and $v_iw$ is not an edge since $u_i$ is inside $C'$ and $w$ is outside $C'$. This is impossible because such a pair of edges is reversed at Step~3.

    Therefore, there is no path between two cycles in $H$. Thus every component $J$ of $H$ has at most one cycle and so $|E(J)| \leq |V(J)|$. Hence $m_4 = |E(H)| \leq |V(H)| = n$.
\end{itemize}

Now, 
\begin{eqnarray*}
  N  & = & \frac{1}{3}m_1 + \frac{1}{2} (m_2 +m_3) + m_4 \\
     & = & \frac{1}{3} (m_1+m_2+m_3 +m_4) + \frac{1}{6}(m_2+m_3+m_4)  + \frac{1}{2}m_4 \\
     & \leq & \frac{1}{3} (3n-6) + \frac{1}{6}(2n-4)  + \frac{1}{2} n  \\
   & \leq & \frac{11}{6} n - \frac{8}{3}
\end{eqnarray*}
Thus $\id^{\leq 3}(G) \leq \frac{11}{6}n - \frac{8}{3}$.

Assume now that $p=4$. Then Step 3+ is also performed, giving more structure on $H$.

\begin{itemize}
    \item Every two no adjacent edges in $H$ are linked by an edge in $G$ as Step 3+ does not apply. Thus contracting in $G$ the edges of a matching in $H$ yields a clique. Since $G$ has no $K_5$-minor because it is planar, there is no matching of size $5$ in $H$. 

    Let $M$ be a maximum matching of $H$ and let $A$ be set the of vertices covered by $M$. Then $|M| \leq 4$, $|A| \leq 8$ and $H - A$ is a stable set. Let $uv \in M$. Towards a contradiction suppose that both $u$ and $v$ have one neighbour, say $w_1$ and $w_2$ respectively, in $V(H) \setminus A$. As $H$ is triangle-free (because of Step 1)  $w_1 \neq w_2$. Thus $M \setminus \{uv\} \cup \{w_1v, w_2u\}$ is a matching of $H$ bigger than $M$, a contradiction. Thus there are at most $|A|$ + $|M|$ vertices in $H$ with degree at least $1$. By the previous argument, for each connected component $J$ of $H$, $|E(J)|\leq |V(J)|$, which implies that $m_4=|E(H)| \leq |A| + |M| \leq 8 + 4 = 12$. 
    
\end{itemize}

Now, 
\begin{eqnarray*}
  N  & = & \frac{1}{3}m_1 + \frac{1}{2} (m_2 +m_3) + m_4 \\
     & = & \frac{1}{3} (m_1+m_2+m_3 +m_4) + \frac{1}{6}(m_2+m_3+m_4)  + \frac{1}{2}m_4 \\
     & \leq & \frac{1}{3} (3n-6) + \frac{1}{6}(2n-4)  + 6  \\
   & \leq & \frac{4}{3} n + \frac{10}{3}
\end{eqnarray*}
Thus $\id^{\leq 5}(G) \leq\id^{\leq 4}(G) \leq \frac{4}{3} n + \frac{10}{3}$.
  \end{proof}

\subsection{Lower bound on \texorpdfstring{$\conv^{\leq 3}_{\cal P}(n)$}{conv(<=3)\_P(n)}.}
%%%%%%%%%%%%%%%%%%%%%%%%%%%%%%%%%%%%%%%%%%%%%%%%%%%%%%%%

Every planar graph of order $n\geq 3$ has at most $3n-6$ edges.
Thus, for $p\geq 3$, an $(\leq p)$-inversion reverses at most $3p-6$ edges of a planar graph.
Hence for every maximal planar graph of order $n$, we have
 $\conv^{\leq p}(G) \geq \frac{3n-6}{3p-6}$,
 thus $\conv^{\leq p}_{\cal P}(n) \geq \frac{3n-6}{3p-6}$.

 The aim of this subsection is to improve on this bound.

For every graph $G$, we  denote the number of vertices with odd degree in $G$ by $n_o(G)$, or simply $n_o$ when $G$ is clear from the context.  

\begin{lemma}\label{lem:odd-degrees}
    Let $G$ be a graph of order $n$ with $n_o$ vertices of odd degree. Then, 
    $\conv^{\leq 3}(G) \geq \frac{|E(G)|}{3} + \frac{n_o}{6}$.
\end{lemma}
\begin{proof}
Let $\vec{G}$ be an orientation of $G$ and $\cev{G}$ be the converse of $\vec{G}$.
Let $\mathcal{X}$ be a set of $(\leq 3)$-inversions that transforms $\vec{G}$ into
$\cev{G}$.
Since each $X \in \mathcal{X}$ contains at most two edges incident to vertex, each vertex $v$ is in at least $\lceil d(v)/2\rceil$ sets of $\mathcal{X}$.
Hence $3|\mathcal{X}| \geq \sum_{v\in V(G)} \lceil d(v)/2\rceil = \sum_{v\in V(G)} d(v)/2 + \frac{n_o}{2} = |E(G)| + \frac{n_o}{2}$. So $|\mathcal{X}| \geq \frac{|E(G)|}{3} + \frac{n_o}{6}$.
\end{proof}

\begin{proposition}\label{prop:triang-odd}
    For any positive integer $q$, there exists a plane triangulation on $4q$ vertices in which all vertices have odd degree.
\end{proposition}
\begin{proof}
    We prove the result by induction on $q$, with $K_4$ being the desired graph for $q=1$.
    
    Assume that we have the desired graph $G$ for some $q$.
    Consider a $3$-face $u_1u_2u_3$ of $G$ and add inside it a disjoint $K_4$ with outer face $v_1v_2v_3$ and the edges $u_iv_j$ for all $i\neq j$.  
    One easily checks that the resulting graph is a plane triangulation in which all vertices have odd degree.
\end{proof}

Since a plane triangulation of order $n$ has $3n - 6$ edges, the previous proposition and Lemma~\ref{lem:odd-degrees} directly imply the following.

\begin{corollary}
    For every $n\equiv 0 \mod 4$, there is a planar graph $G$ of order $n$ such that
    $\conv^{\leq 3}(G) \geq \frac{7n}{6}-2$.
\end{corollary}

\medskip

Let $n$ be an integer with $n \geq 5$.
The {\bf double wheel} of order $n$, denoted by $DW_n$ is the graph obtained from a cycle of order $n-2$ by adding two vertices adjacent to all vertices of the cycles.
A double wheel is clearly planar. Moreover if $n>p\geq 4$, an induced subgraph on $p$ vertices has at most $3p-7$ edges.
Hence, for all $p\geq 4$,
$\conv^{\leq p}_{\cal P}(n) \geq \conv^{\leq p}(DW_n) \geq \frac{3n-6}{3p-7}$.

\begin{proposition}
    Let $p$ be an integer with $p \geq 3$.
    For every integer $n$ with $n \geq p+2$,
    \[
        \conv^{\leq p}(DW_n) \geq \frac{p-1}{(p-2)^2+1} \cdot (n-2).
    \]
\end{proposition}

\begin{proof}
    Let $n$ be an integer with $n \geq p+2$.
    We denote by $u_1, u_2$ the two vertices of degree $n-2$ in $DW_n$.
    Let $w \colon E(DW_n) \to \mathbb{R}$ be defined as follows
    for every $e \in E(DW_n)$,
    \[
        w(e) = 
        \begin{cases}
            \frac{1}{(p-2)^2+1} & \textrm{if $e$ is incident to $u_1$ or $u_2$,} \\
            \frac{p-3}{(p-2)^2+1} & \textrm{if $e$ is not incident to $u_1$ and $u_2$.} \\
        \end{cases}
    \]
    Then, it is straightforward to check that $\sum_{e \in E(DW_n \langle X \rangle)} w(e) \leq 1$ for every $(\leq p)$-set $X \subseteq V(DW_n)$.
    On the other hand, $\sum_{e \in E(DW_n)} w(e) = \frac{p-1}{(p-2)^2+1} \cdot (n-2)$.
    We deduce that $\conv^{(\leq p)}(DW_n) \geq \frac{p-1}{(p-2)^2+1} \cdot (n-2)$.
\end{proof}

\section{Open problems}

In Section~\ref{sec:upper-gen}, we proved that $\id^{\leq p}(G) \leq \left\lceil\frac{|E(G)|}{\lfloor p/2\rfloor}\right \rceil + \frac{1}{2} p^2$. 
Since for a matching graph, $\conv^{\leq p}(G) = \left\lceil\frac{|E(G)|}{\lfloor p/2\rfloor}\right \rceil \leq \id^{\leq p}(G)$, a natural question 
is to determine the smallest number $\Psi_p$, (resp. $\Psi'_p$) such that 
$\id^{\leq p}(G)\leq \left\lceil \frac{|E(G)|}{\lfloor p/2\rfloor}\right \rceil + \Psi_p$
(resp. $\conv^{\leq p}(G) \leq  \left\lceil \frac{|E(G)|}{\lfloor p/2\rfloor}\right \rceil + \Psi'_p$) 
for every graph $G$.
As proved in~\cite{HHR24}, if $n\leq p$, then $\id^{\leq p}(K_n) = \id(K_n) = n-1$. So $\Psi_p \geq n-1 - \left\lceil \frac{\binom{n}{2}}{\lfloor p/2\rfloor}\right \rceil \geq n-1 - \left\lceil \frac{n(n-1)}{p-1} \right\rceil$.
For $n=\lceil p/2 \rceil$, we obtain $\Psi_p \geq \frac{1}{4}p -\frac{3}{2}$.
We believe that this lower bound is tighter than the upper bound.

\begin{problem}
 Does there exist a constant  $C$ such that $\Psi_p \leq C\cdot p$ for every integer $p$ greater than $1$?
\end{problem}

For planar graphs, it would be interesting to close the gaps between the upper and lower bounds proved on Section~\ref{sec:planar}. In particular, for outerplanar graphs we believe that better bounds can be obtained. Recall that every outerplanar graph $G$ is $2$-degenerate, so by Corollary~\ref{cor:procedure-greedy}, $\id^{\leq p}(G) \leq |V(G)| -1$ for every $p \geq k+1$. 

\begin{problem}
 Does there exist a constant $\alpha <  1$ such that $\id^{\leq p}(G) \leq \alpha n$ for every outerplanar graph $G$ of order $n$?
\end{problem}

\bibliographystyle{alpha}
\bibliography{bibliography}

\end{document}